\documentclass[11pt,a4paper]{article}

\usepackage{amssymb,amsmath}
\usepackage[normal]{subfigure}
\usepackage{epsfig}
\usepackage{graphicx,color}
\usepackage[latin1]{inputenc}
\usepackage{hyperref}

\newtheorem{theorem}{Theorem}[section]

\newtheorem{corollary}[theorem]{Corollary}
\newtheorem{proposition}[theorem]{Proposition}
\newtheorem{definition}[theorem]{Definition}

\newtheorem{remark}[theorem]{Remark}
\newtheorem{example}[theorem]{Example}

\newenvironment{proof}{{\it Proof.} }

\def\E{\mathbb{E}}

\def\cov{\mathrm{Cov}}

\def\phi{\varphi}
\newcommand {\nn}{\nonumber}
\newcommand {\noi}{\noindent}

\def\mathsf{\bf}

\def\R{\mathbb{R}}
\def\Z{\mathbb{Z}}

\def\E{\mathrm E}
\def\P{\mathrm P}

\def\text{\mbox}

\def\1{{\bf 1}}

\newcommand\beqn{\begin{displaymath}}  
\newcommand\eeqn{\end{displaymath}}

\begin{document}
\title{Projective stochastic equations and nonlinear long memory}
\title{Projective stochastic equations and nonlinear long memory
}
\author{Ieva Grublyt\.e and Donatas Surgailis}
\date{\today \\  \small Vilnius University}
\maketitle

\begin{abstract}
A projective moving average $\{ X_t, t \in \Z \}$   is a Bernoulli shift  written as
a backward martingale transform of the innovation
sequence. We introduce a new class of nonlinear stochastic equations for projective moving averages, termed projective equations,
involving a (nonlinear) kernel $Q$ and a linear combination of projections of $X_t$
on "intermediate" lagged innovation subspaces with given coefficients $\alpha_i, \beta_{i,j}$. The class of
such equations include usual moving-average processes and the Volterra series of
the LARCH model. Solvability of projective equations is studied, including a
nested Volterra series representation of the solution $X_t$. We show that under natural conditions on
$Q, \alpha_i, \beta_{i,j}$, this solution exhibits covariance and distributional long memory, with
fractional Brownian motion as the limit of the corresponding partial sums process.

\end{abstract}

\section{Introduction}
A discrete-time second-order  stationary process  $\{X_t, t \in \Z\}$ is called  {\it long memory}
if its covariance $\gamma(k) = {\rm cov}(X_0,X_k)$ decays slowly with the lag in such a way that
its absolute series diverges: $\sum_{k=1}^\infty |\gamma(k)| = \infty $. In the converse case
when $\sum_{k=1}^\infty |\gamma(k)| < \infty $ and $\sum_{k=1}^\infty \gamma(k) \neq  0 $ the process
$\{X_t\}$ is said  {\it short memory}.  Long memory processes
have different properties from short memory (in particular, i.i.d.) processes.
Long memory processes have been found to arise in a variety of physical and social sciences. See, e.g.,
the monographs
Beran~\cite{ref2}, Doukhan et al.~\cite{ref6}, Giraitis et al.~\cite{ref11} and the references therein.

Probably, the most important model of long memory processes is the linear, or moving average process
\begin{eqnarray} \label{linear}
X_t&=&\sum_{s\le t} b_{t-s} \zeta_s, \qquad t \in \Z,
\end{eqnarray}
where $\{\zeta_s, s\in \Z\}$ is a standardized i.i.d. sequence,
and the moving average coefficients $b_j$'s decay slowly so that  $\sum_{j=0}^\infty |b_j| = \infty,\,
\sum_{j=0}^\infty b^2_j < \infty $. The last condition guarantees that the series in (\ref{linear}) converges in mean
square and satisfies $\E X_t = 0,\, \E X^2_t = \sum_{j=0}^\infty b^2_j < \infty$. In the literature it is often
assumed that the coefficients regularly decay as
\begin{equation}\label{bj}
b_j \sim \kappa j^{d-1}, \qquad j \to \infty, \qquad \exists \  \kappa >0, \  d \in (0,1/2).
\end{equation}
Condition (\ref{bj}) guarantees that
\begin{equation}\label{covk}
\gamma(k)\  = \ \sum_{j=0}^\infty b_j b_{k+j} \ \sim \  \kappa^2 B(d,1-2d) k^{1-2d}, \qquad k \to \infty
\end{equation}
and hence
$\sum_{k=1}^\infty |\gamma(k)| = \infty. $ The parameter $d$ in (\ref{bj}) is called the {\it long memory parameter} of
$\{X_t\}$. A particular case of linear processes
(\ref{linear})-(\ref{bj}) is the parametric class ARFIMA$(p,d,q)$, in which case
$d \in (0, 1/2) $ is the order of fractional integration. An important property
of the linear process in (\ref{linear})-(\ref{bj}) is the fact that its (normalized) partial sums process
$S_n(\tau) := \sum_{j=1}^{[nt]} X_j, \, \tau \ge 0$ tends to a fractional Brownian motion (\cite{ref4}), viz.,
\begin{equation}\label{dalX}
n^{-d - 1/2} S_n(\tau) \  \to_{D[0,1]} \  \sigma(d) B_H(t),
\end{equation}
where $H = d + \frac{1}{2} $ is the Hurst parameter, $\sigma(d)^2 := \kappa^2 B(d,1-2d)/d(1+2d) >0 $ and $\to_{D[0,1]}$ denotes the
weak convergence of random processes in the Skorohod space $D[0,1]$.

On the other hand, the linear model (\ref{linear}) has its drawbacks and sometimes is not capable
of incorporating empirical features (``stylized facts'') of some observed time series. The
"stylized facts" may include typical asymmetries, clusterings, and other nonlinearities
which are often observed in financial data, together with long memory.

The present paper introduces a new class of nonlinear processes which generalize the linear model
in (\ref{linear})-(\ref{bj}) and enjoy similar long memory properties to (\ref{covk})
and (\ref{dalX}). These processes are defined through solutions of the so-called
{\it projective stochastic equations}. Here, the term  ``projective'' refers to the fact that these equations
contain linear combinations of projections, or conditional expectations, of $X_t$'s on lagged innovation subspaces which enter the equation in a nonlinear way.

\smallskip

Let us explain the main idea of our construction. We call a {\it projective moving average} a random process $\{X_t\}$ of the form
\begin{eqnarray} \label{proma0}
X_t&=&\sum_{s\le t} g_{s,t} \zeta_s, \qquad t \in \Z,
\end{eqnarray}
where $\{\zeta_s\} $ is a sequence of standardized i.i.d. r.v.'s as in (\ref{linear}), $g_{t,t} \equiv g_0$ is a deterministic constant 
and  $g_{s,t}, \, s< t$ are r.v.'s
depending only on $\zeta_{s+1}, \dots, \zeta_t$ 
such that
\begin{equation} \label{gst0}
g_{s,t} = g_{t-s}(\zeta_{s+1}, \dots, \zeta_t), \qquad s < t,
\end{equation}
where
$g_j: \R^j \to \R, \, j=1,2, \dots$ are {\it nonrandom} functions satisfying
\begin{equation}\label{gsum}
\sum_{s \le t} \E g^2_{s,t}\  = \ \sum_{s \le 0} \E g^2_{-s}(\zeta_{s+1}, \dots, \zeta_0) \ < \ \infty.
\end{equation}
It follows easily that under condition  (\ref{gsum}) the series in (\ref{proma0}) converges in mean square and define
a stationary process with zero mean and finite variance $\E X^2_t = \sum_{s \le t} \E g^2_{s,t}$. The next question -
how to choose the ``coefficients'' $g_{s,t}$  (\ref{gst0}) so that they depend on $X_t$ and behave like
(\ref{bj}) when $j= t-s \to \infty $?

A particularly simple choice of the $g_{s,t}$'s  to achieve the above goals is
\begin{equation} \label{Qst}
g_{s,t} \ =\   b_{t-s}  Q(  \E_{[s+1,t]} X_t), \qquad s \le t
\end{equation}
where $b_j$ are as in (\ref{bj}), $Q: \R \to \R$ is a given deterministic kernel,
and $\E_{[s+1,t]} X_t := \E [X_t | \zeta_v, s+1 \le v \le t]$ is the projection of $X_t$ onto the subspace of $L^2 $ generated
by the innovations $\zeta_v, s+1\le  v\le t $ (the conditional expectation).
The corresponding projective stochastic equation has the form
\begin{eqnarray} \label{proma1}
X_t&=&\sum_{s\le t}  b_{t-s}  Q(  \E_{[s+1,t]} X_t) \zeta_s.
\end{eqnarray}
Notice  that when $s \to -\infty $ then $ \E_{[s+1,t]} X_t \to  X_t$ by a general property of a conditional expectation
and then $g_{s,t} \sim b_{t-s}  Q(X_t)$ if $Q$ is continuous. This means that
the $g_{s,t}$'s in (\ref{Qst}) feature 
both the long memory in 
(\ref{bj}) and the dependence 
on the ``current'' value $X_t$ through  $Q(X_t)$.
In particular, for $Q(x) = \max(0,x)$, the behavior of $g_{s,t}$ in (\ref{Qst})
strongly depends on the sign of $X_t$ and
the trajectory of (\ref{proma1}) appears very asymmetric (see
Fig.~\ref{lmQint2}, top).

Let us briefly describe the remaining sections.
Sec.~2 contains basic definitions and properties of  projective processes.
Sec.~3 introduces the notion of nested Volterra series which plays an important role for solving of
projective equations.
Sec.~4 introduces a general class of projective stochastic equations, (\ref{proma1}) being a particular case.
We obtain sufficient conditions
of solvability of these equations, and a recurrent formula for computation of "coefficients"
$g_{s,t}$ (Theorem \ref{thm1}). Sec.~5 and 6 present some examples and
simulated trajectories and histograms of  projective equations.
It turns out that the LARCH model studied in \cite{ref7} and
elsewhere is a particular case of 
projective equations corresponding
to linear kernel $Q(x)$ (Sec.~5). Some modifications of projective equations are discussed in Sec.7.   
Sec.~8 deals with long memory properties of stationary solutions of stochastic projective equations.
We show that under some additional conditions these solutions have
long memory properties similar to  (\ref{covk}) 
and (\ref{dalX}).

Finally, we remark that  ``nonlinear long memory''  is a
general term and that other time series models different from ours
for such behavior
were proposed in the literature. Among them, probably the most studied
class are subordinated processes of the form $\{Q(X_t)\}$, where
$\{X_t\}$ is a Gaussian or linear long memory process and $Q: \R \to \R$ is a nonlinear function.
See \cite{ref19}, \cite{ref14} and \cite{ref11} for a detailed discussion. A related class
of Gaussian subordinated stochastic volatility models is studied in \cite{ref16}. \cite{ref6a} discuss a class
of long memory Bernoulli shifts.
\cite{ref1} consider
fractionally integrated process with nonlinear autoregressive innovations.
A general invariance principle for fractionally integrated models with weakly dependent innovations satisfying
a projective dependence condition of \cite{ref4a}
is established in \cite{ref17}. See also \cite{ref21} and Remark \ref{Wu0} below.

We expect that the results of this paper can be extended in several directions, e.g., 
projective equations with initial condition, continuous time processes, random field set-up,
infinite variance processes.
For applications,
a major challenge is estimation of ``parameters'' of projective equations. We plan to study some
of these  questions in the future.

\section{Projective processes and their properties}

Let $\{\zeta_t, t \in {\Z}\} $ be a sequence of i.i.d. r.v.'s with $\E \zeta_0 = 0, \, \E \zeta^2_0 = 1$.
For any integers $s\le t $ we denote
${\mathcal F}_{[s,t]} := \sigma\{\zeta_u: u \in [s,t]\} $ the sigma-algebra generated by $\zeta_u, u \in [s,t]$, \,
${\mathcal F}_{(-\infty,t]} := \sigma\{\zeta_u: u \le t\}, \,
{\mathcal F} := \sigma\{\zeta_u: u \in \Z\}. $ For $s > t$,
we define ${\mathcal F}_{[s,t]} := \{\emptyset, \Omega \} $ as  the trivial sigma-algebra.  Let
$L^2_{[s,t]}, \, L^2_{(-\infty,t]}, \,
L^2$ be the spaces of all square integrable r.v.'s $\xi $ measurable w.r.t.  ${\mathcal F}_{[s,t]}, \,
{\mathcal F}_{(-\infty,t]}, \,
{\mathcal F}$, respectively.
For any $s, t \in \Z$ let
$$
\E_{[s,t]} [\xi ] := \E\left[ \xi \big| {\mathcal F}_{[s,t]} \right], \qquad \xi \in  L^2
$$
be the conditional expectation. Then $\xi \mapsto \E_{[s,t]} [\xi ] $ is a bounded linear operator in  $L^2$; moreover,
$\E_{[s,t]}, \, s,t \in \Z $ is a projection family satisfying $\E_{[s_2, t_2]} \E_{[s_1,t_1]} = \E_{[s_2, t_2]\cap [s_1,t_1]} $
for any intervals $[s_1,t_1], [s_2,t_2] \subset \Z$.
From the definition of conditional expectation it follows that if $g_u:  \R \to \R, u \in \Z $ are arbitrary measurable functions with $\E g^2_u(\zeta_u) < \infty, $
$[s_2,t_2] \subset \Z$ is a given interval  and
$\xi = \prod_{u\in [s_2,t_2]} g_u(\zeta_u)$ is a product of independent r.v.'s, then for any interval $[s_1, t_1] \subset \Z$
\begin{eqnarray*}
\E_{[s_1,t_1]} \prod_{u\in [s_2,t_2]} g_u(\zeta_u)&=&\prod_{u\in [s_1,t_1] \cap [s_2,t_2]} g_u(\zeta_u)
\prod_{v \in [s_2,t_2]\setminus [s_1,t_1]} \E [g_v(\zeta_v)].
\end{eqnarray*}
In particular, if $\E g_u(\zeta_u) = 0, \, u \in \Z $ then
\begin{eqnarray} \label{projP}
\E_{[s_1,t_1]} \prod_{u\in [s_2,t_2]} g_u(\zeta_u)&=&\begin{cases}
\prod_{u\in [s_1,t_1]} g_u(\zeta_u), &[s_2,t_2] \subset [s_1,t_1], \\
0, &[s_2,t_2] \not\subset [s_1,t_1].
\end{cases}
\end{eqnarray}
Any r.v. $Y_t \in L^2_{(-\infty,t]} $ can be expanded into orthogonal series
$Y_t  =  \E Y_t + \sum_{s \le t} P_{s,t} Y_t, $ where $P_{s,t} Y_t := (\E_{[s,t]} - \E_{[s+1,t]})Y_t.$
Note that  $\{ P_{s,t} Y_t, {\cal F}_{s,t}, s\le t\}$ is a backward martingale difference sequence and
$\E Y^2_t  =   (\E Y_t)^2 +  \sum_{s \le t} \E (P_{s,t} Y_t)^2. $

\begin{definition} \label{propro}
A {\it projective process}
is a random sequence $\{Y_t \in  L^2_{(-\infty,t]},  \, t \in \Z\}$ of the form
\begin{equation}
Y_t\  = \  \E Y_t +  \sum_{s\le t} g_{s,t} \zeta_s, \label{XMA}
\end{equation}
where $g_{s,t} $ are r.v.'s satisfying the following conditions (i) and (ii):

\smallskip

\noindent (i) $g_{s,t}$ is ${\mathcal F}_{[s+1,t]}$-measurable, \ $ \forall s, t \in \Z, \, s < t; \, g_{t,t} $ is a
deterministic number;

\smallskip

\noindent (ii) $\sum_{s \le t} \E g^2_{s,t} < \infty, \ \forall \, t \in \Z. $

\end{definition}

In other words, a projective process has the property that the projections $\E_{[s,t]} Y_t = \E Y_t + \sum_{i=s}^{t} P_{i,t} Y_t =
\E Y_t + \sum_{i=s}^t \zeta_i g_{i,t}, \, s \le t $ form a backward martingale transform
w.r.t. the nondecreasing family ${\{\cal F}_{[s,t]}, s \le t\} $ of sigma-algebras, for each $t \in \Z$ fixed. A consequence
of the last fact is the following moment inequality which is an easy consequence of
Rosenthal's inequality (\cite{ref12}, p.24).
See also \cite{ref11}, Lemma 2.5.3.

\begin{proposition} \label{propmom} Let $\{Y_t \} $ be a projective process in (\ref{XMA}).
Assume that $\mu_p := \E |\zeta_0|^p < \infty $ and $\sum_{s \le t} (\E |g_{s,t}|^p )^{2/p} < \infty $
for some $p \ge 2$. Then $\E |Y_t|^p < \infty $. Moreover,
there exists a constant
$C_p < \infty $ depending on $p$  alone and such that
\begin{equation*}
\E |Y_t|^p \ \le \ C_p \Big( |\E Y_t|^p + \mu_p \big(\sum_{s \le t} (\E |g_{s,t}|^p )^{2/p} \big)^{p/2}\Big).
\end{equation*}
\end{proposition}

\begin{definition}  \label{proma} A {\it projective
moving average} is a projective process of (\ref{XMA}) such that the mean $\E Y_t = \mu $ is constant and
there exist a number $g_0 \in \R$ and
nonrandom measurable functions $g_j: \R^j \to \R, j=1,2, \dots $ such that
$$
g_{s,t} = g_{t-s}(\zeta_{s+1}, \dots, \zeta_t)  \qquad \text{a.s., \ for any} \quad  s \le t, \, s,t \in \Z.
$$
\end{definition}

\smallskip

By definition,  a projective moving average is a stationary Bernoulli shift (\cite{ref5}, p.21):
\begin{equation}\label{bernoulli}
Y_t \ = \  \mu +  \sum_{s \le t} \zeta_s g_{t-s} (\zeta_{s+1}, \dots, \zeta_t)
\end{equation}
with  mean $\mu$ and covariance
\begin{eqnarray}
\hskip-.3cm \cov (Y_s, Y_t)&=&\sum_{u\le s} \E [g_{s-u } (\zeta_{u+1}, \dots, \zeta_s) g_{t-u}(\zeta_{u+1}, \dots, \zeta_t) ]  \nn \\
&=& \sum_{u\le 0} \E[ g_{-u} (\zeta_{u+1}, \dots, \zeta_0) g_{t-s -u}(\zeta_{u+1}, \dots, \zeta_{t-s-u}) ], \quad s\le t.\label{XMA1}
\end{eqnarray}
These facts together with the
ergodicity of Bernoulli shifts  (implied by a general result in \cite{ref18}, Thm.3.5.8) are summarized in the following corollary.

\smallskip

\begin{corollary} \label{rem1} A projective moving average is a strictly stationary and ergodic
stationary process with finite variance and
covariance given in (\ref{XMA1}).
\end{corollary}

\begin{remark} \label{rem2} {\rm If the coefficients  $g_{s,t}$ are nonrandom, a projective moving average is a  linear process
$Y_t = \mu + \sum_{s\le t} g_{t-s} \zeta_s, \, t \in \Z.$}
\end{remark}

\begin{proposition} \label{propfilter} Let  $\{Y_t\}$ \ be a projective process of (\ref{XMA}) and $\{a_j, j \ge 0\}$ a deterministic
sequence, $\sum_{j=0}^\infty |a_j| < \infty, \sum_{j=0}^\infty |a_j| |\E Y_{t-j}| < \infty $.
Then $\{u_t := \sum_{j=0}^\infty a_j Y_{t-j}, \, t \in \Z\} $  is a projective
process $u_t = \E u_t + \sum_{s \le t} \zeta_s G_{s,t} $  with $\E u_t = \sum_{j=0}^\infty a_j \E Y_{t-j}$ and
coefficients $G_{s,t} := \sum_{j=0}^{t-s} a_j g_{s,t-j}$.
\end{proposition}

\noi {\it Proof} follows easily by the Cauchy-Schwarz inequality  and is omitted. \hfill   $\Box$

\begin{proposition} \label{propunique} If \ $\{Y_t\}$ \ is a projective process of (\ref{XMA}), then for any $s\le t$
\begin{eqnarray}\label{Eg}
\hskip-.3cm \E_{[s,t]} Y_t&=&\E Y_t + \sum_{s \le u \le t} \zeta_u g_{u,t}, \quad
P_{s,t} Y_t \ = \ (\E_{[s,t]} - \E_{[s+1,t]}) Y_t\  =\  \zeta_s g_{s,t}.
\end{eqnarray}
The representation (\ref{XMA}) is unique: if
(\ref{XMA}) and $Y_t = \sum_{s\le t} g'_{s,t} \zeta_s $ are two representations, with $g'_{s,t}$ satisfying
conditions (i) and (ii) of Definition \ref{propro}, then $g'_{s,t} = g_{s,t}\ \forall \  s\le t. $
\end{proposition}

\begin{proof}
(\ref{Eg}) is immediate by definition of projective process. From
(\ref{Eg}) it follows that $\zeta_s
g''_{s,t} = 0,$ where $g''_{s,t} := g_{s,t} - g'_{s,t}$ is independent of $\zeta_s$. Relation
$\E \zeta^2_s = 1$ implies $\P (|\zeta_s|^2 > \epsilon) >0$ for all $\epsilon >0 $ small enough.
Hence,  $0= \P (|\zeta_s g''_{s,t}|> \epsilon)
\ge \P (|\zeta_s| > \sqrt{\epsilon}, |g''_{s,t}|> \sqrt{\epsilon}) =
\P (|\zeta_s| > \sqrt{\epsilon}) \P( |g''_{s,t}|> \sqrt{\epsilon}),$ implying
$ \P( |g''_{s,t}|> \sqrt{\epsilon}) = 0$ for any $\epsilon >0$. \hfill $\Box$
\end{proof}

The following invariance principle is due to Dedecker and Merlev\`ede (\cite{ref4a}, Cor.~3), see also (\cite{ref20}, Thm.~3 (i)).

\begin{proposition} Let  $\{Y_t\}$ \ be a projective moving average of (\ref{XMA}) such that $\mu = 0$
and
\begin{equation}\label{DedMer}
\Omega(2) \ := \ \sum_{t=0}^\infty \| g_{0,t} \| \ < \infty,
\end{equation}
where $\| \xi \| = \E^{1/2} [\xi^2], \, \xi \in L^2.$
Then
\begin{equation} \label{DedMer1}
n^{-1/2} \sum_{t=1}^{[n\tau]}Y_t\, \longrightarrow_{D[0,1]}\, c_Y B(\tau),
\end{equation}
where $B$ is a standard Brownian motion and $c^2_Y := \| \sum_{t=0}^\infty g_{0,t} \|^2 = \sum_{t\in \Z} \E [Y_0 Y_t]. $
\end{proposition}

\section{Nested Volterra series }


First we introduce some notation. Let $T \subset \Z $ be a set of integers which is bounded from above, i.e.,  $\sup \{s: s \in T\} < \infty $.
Let
${\cal S}_T $ be a class of finite nonempty subsets  $S  = \{s_1, \dots, s_n \} \subset T, \, s_1 < \dots < s_n, \, n\ge 1$.
Write $|S |$ for the cardinality of $S \subset Z$.
For any
$S = \{s_1, \dots, s_n\} \in {\cal S}_T, \, S' = \{s'_1, \dots, s'_m \} \in {\cal S}_T,
$ the notation
$S \prec S' $ means that $m= n+1 $ and $s_1= s'_1, \dots, s_n =s'_n < s'_{n+1} = s'_m $. In particular,
$S \prec S' $ implies $S \subset S'$ and $|S'\setminus S| = 1$.  Note that $\prec$ is not a partial order in
${\cal S}_T $ since $S \prec S', \, S'\prec S'' $ do not imply $S \prec S'' $.
A set $ S \in {\cal S}_T $ is said {\it maximal} if there is no $S' \in {\cal S}_T $ such that
$S \prec S'$.
Let ${\cal S}^{\rm max}_T $ denote the class of all maximal elements of ${\cal S}_T. $

\smallskip

\begin{definition}  \label{nested} Let $T$ and
${\cal S}_T $ be as above. Let
${\cal G}_T := \{G_S,   S \in {\cal S}_T  \} $ be a family of measurable functions $G_S = G_{s_1, \dots, s_m}: \R \to \R$
indexed by sets $ S = \{s_1, \dots, s_m \} \in {\cal S}_T $ and such that
$G_S =: a_S$ is a constant function for any maximal set $S \in {\cal S}^{\rm max}_T $.
A {\it nested Volterra series} is a sum
\begin{eqnarray}
\hskip-1cmV({\cal G}_T)&=&\sum_{S_1 \in {\cal S}_T: |S_1|=1} \zeta_{S_1} G_{S_1} \Big( \sum_{S_1\prec S_2} \zeta_{S_2\setminus S_1} G_{S_2}
\Big( \dots \zeta_{S_{p-1}\setminus S_{p-2}}\nn \\
&&\hskip3cm  \times G_{S_{p-1}} \Big(\sum_{S_{p-1}\prec S_p }
\zeta_{S_p\setminus S_{p-1}} G_{S_p}
\Big)
\Big) \Big), \label{nestvolt0}
\end{eqnarray}
where the nested summation is taken over all sequences $S_1 \prec S_2 \prec \dots \prec S_p \in {\cal S}^{\rm max}_T, p=1,2,\dots$,
with the convention
$G_S  = a_S, \, S \in {\cal S}^{\rm max}_T, $ and  $\zeta_S := \zeta_s $ for $S = \{s\}, |S|=1.$

In particular, when ${\cal S}_T = \{ S: S \subset T\} $ is the class of all subsets of $T$,
(\ref{nestvolt0}) can be rewritten as
\begin{eqnarray}
\hskip-.5cm V({\cal G}_T)&=&\sum_{s_1 \in  T} \zeta_{s_1} G_{s_1} \Big( \sum_{s_1< s_2 \in T} \zeta_{s_2} G_{s_1,s_2}
\Big( \dots \zeta_{s_{p-1}} \nn \\
&&\hskip3cm \times G_{s_1,\dots,s_{p-1}}\Big(\sum_{s_{p-1} < s_p \in T} \zeta_{s_p} G_{s_1, \dots, s_p}
\Big) \Big) \Big), \label{nestvolt}
\end{eqnarray}
where the last sum is taken over all $s_p > s_{p-1}$ such that  $\{s_1, \dots, s_p\} \in {\cal S}^{\rm max}_T.$
\end{definition}

The following example clarifies the above definition and its relation to the usual Volterra series
(\cite{ref5}, p.22).

\smallskip

\begin{example} \label{exnest} Let $ T(t) = (-\infty, t] \cap \Z, \, t \in \Z$  and ${\cal S}_{T(t)} $ be the class of all subsets
$S = \{s_1,\dots, s_k \} \subset T(t) $ having $k$ points. Let ${\cal G}_{T(t)} = \{G_S,   S \in {\cal S}_{T(t)}  \} $ be a family of
linear functions
$$
G_{S}(x):=
\begin{cases}
x,&S \in {\cal S}_T, S \not \in {\cal S}^{\rm max}_{T(t)}, \\
a_S = a_{s_1,\dots, s_k}, &S  =\{s_1,\dots, s_k\}  \in {\cal S}^{\rm max}_{T(t)}.
\cr
\end{cases}
$$
Then
\begin{eqnarray} \label{volt}
V({\cal G}_{T(t)})&=&\sum_{s_1 < \dots < s_k  \le   t} a_{s_1, \dots, s_k}  \zeta_{s_1} \zeta_{s_2} \cdots \zeta_{s_k}  \
= \ \sum_{S \subset T, |S| = k} a_S \zeta^S,
\end{eqnarray}
$\zeta^S := \zeta_{s_1} \zeta_{s_2} \cdots \zeta_{s_k},$
is the (usual) Volterra series of order $k$.   The series
in (\ref{volt}) converges in mean square if and only if
\begin{equation} \label{volt1}
A_{T(t)}\ :=\ \sum_{s_1 < \dots < s_k  \le   t} a^2_{s_1, \dots, s_k}  < \infty,
\end{equation}
in which case $\E V^2({\cal G}_{T(t)})= A_{T(t)}, \, \E V({\cal G}_{T(t)}) = 0$.
\end{example}

\begin{proposition}\label{propvolt} Let $ T(t) := (-\infty, t] \cap \Z, \, t \in \Z$ as in Example \ref{exnest}.
Assume that  the system ${\cal G}_{T(t)} = \{G_S,   S \in {\cal S}_{T(t)}\} $ in Definition \ref{nested}
satisfies
the following condition
\begin{equation}\label{GSbdd}
|G_S(x)|^2\ \le \ \begin{cases}
\alpha^2_S + \beta^2_S x^2, &S \in {\cal S}_{T(t)}, \, S \not\in {\cal S}^{\rm max}_{T(t)}, \\
\alpha^2_S(= a^2_S), &S \in {\cal S}^{\rm max}_{T(t)},
\end{cases}
\end{equation}
where $\alpha_S, \beta_S$ are real numbers satisfying
\begin{eqnarray}
{\cal A}_{T(t)}&:=&
\sum_{p \ge 1} \sum_{S_1 \prec S_2 \prec \dots \prec S_p}
\beta^2_{S_1} \beta^2_{S_2} \cdots \beta^2_{S_{p-1}} \alpha^2_{S_p} \ < \ \infty,
 \label{calAt}
\end{eqnarray}
where the inner sums are taken over all sequences
$S_1 \prec S_2 \prec \dots \prec S_p, \, S_i \in {\cal S}_{T(t)}, \, 1\le i \le p$ with
$|S_1|=1 $ and $S_p  \in {\cal S}^{\rm max}_{T(t)}$.

\smallskip

Then, the nested Volterra series
$V({\cal G}_{T(t)})$ in (\ref{nestvolt}) converges in mean square and satisfies
$\E  V({\cal G}_{T(t)})^2 \le {\cal A}_{T(t)}, \, \E  V({\cal G}_{T(t)}) = 0.$ Moreover,
$ X_t :=  V({\cal G}_{T(t)})$ is a projective process with zero mean and coefficients
\begin{eqnarray}\label{gst}
g_{s,t}&:=&
G_{S_1} \big( \sum_{S_1\prec S_2} \zeta_{S_2\setminus S_1} G_{S_2}
\big(\cdots \zeta_{S_{p-1}\setminus S_{p-2}} \nn\\
&&\hskip3cm \times G_{S_{p-1}} \big(\sum_{S_{p-1}\prec S_p }
\zeta_{S_p\setminus S_{p-1}} G_{S_p}
\big) \big) \big)
\end{eqnarray}
if $S_1= \{s\} \in {\cal S}_{T(t)}$, $g_{s,t}:=0$ otherwise,
where the nested summation is defined as in (\ref{nestvolt0}).
\end{proposition}

\begin{proof}
Clearly, the coefficients $g_{s,t}$ in (\ref{gst}) satisfy the measurability condition (i) of Definition
\ref{propro}. Condition (ii) for these coefficients
follows by recurrent application of  (\ref{GSbdd}):
\begin{eqnarray*}
\sum_{s\le t} \E g^2_{s,t}
&=&\sum_{S_1 \in {\cal S}_{T(t)}: |S_1|=1} \E G^2_{S_1} \big( \sum_{S_1\prec S_2} \zeta_{S_2\setminus S_1} G_{S_2}
(\dots ) \big)\\
&\le& \sum_{S_1 \in {\cal S}_{T(t)}: |S_1|=1} \big( \alpha_{S_1}^2 + \beta_{S_1}^2\E \big(\sum_{S_1\prec S_2} \zeta_{S_2\setminus S_1} G_{S_2}( \dots ) \big)^2\big) \nn \\
&\le& \sum_{S_1 \in {\cal S}_{T(t)}: |S_1|=1} \big( \alpha_{S_1}^2 + \beta_{S_1}^2 \sum_{S_1\prec S_2} \big( \alpha_{S_2}^2+\beta_{S_2}^2 \E \big(\sum_{S_2\prec S_3 } \zeta_{S_3\setminus S_2} G_{S_3}(\dots )
\big)^2 \big) \big)\nn \\
&\le&\sum_{S_1 \in {\cal S}_{T(t)}: |S_1|=1}  \Big(\alpha_{S_1}^2+ \beta_{S_1}^2 \sum_{S_1\prec S_2} \alpha_{S_2}^2+
\beta_{S_1}^2 \sum_{S_1\prec S_2\prec S_3} \beta_{S_2}^2 \alpha_{S_3}^2  +\dots \Big) \nn \\
&=& \sum_{p\ge 1} \sum_{S_1\prec S_2\prec \dots \prec S_p}\beta_{S_1}^2 \beta_{S_2}^2 \cdots \beta_{S_{p-1}}^2 \alpha_{S_p}^2 = {\cal A}_{T(t)} < \infty.
\end{eqnarray*}
Thus,  $ X_t = \sum_{s\le t} g_{s,t} \zeta_s$ is a well-defined projective process and
$X_t =  V({\cal G}_{T(t)})$. \hfill $\Box$
\end{proof}

\smallskip

\begin{remark} \label{rem4} {\rm In the case of a usual Volterra series in (\ref{volt}), condition
(\ref{GSbdd}) is satisfied with $\alpha_S =0, \beta_S =1 $ for $S \in {\cal S}_{T(t)}, \, S \not\in {\cal S}^{\rm max}_{T(t)}, $
and the sums ${\cal A}_{T(t)}$ of (\ref{calAt}) and $A_{T(t)}$ of (\ref{volt1}) coincide:
${\cal A}_{T(t)} = A_{T(t)}. $ This fact confirms that 
condition
(\ref{calAt}) for the convergence of nested Volterra series cannot be generally improved.}
\end{remark}


\section{Projective stochastic equations}

Let $Q_{s,t} = Q_{s,t} (x_{u,v}, s < u \le v  \le t), \, s, t \in \Z,  s < t$ be some given measurable deterministic functions
depending on $(t-s)(t-s+1)/2 $ real variables $x_{u,v},\, s <t,$ and  $\mu_t, \, Q_{t,t}, \, t \in \Z$ be some given constants. A {\it projective stochastic equation} has the form
\begin{eqnarray}\label{proSE}
X_t&=&\mu_t + \sum_{s\le t} \zeta_s Q_{s,t}(\E_{[u,v]} X_v,  s < u \le v \le t).
\end{eqnarray}

\begin{definition}
By {\it solution of (\ref{proSE})} we mean a projective process $\{ X_t, t \in \Z \}$
satisfying 
$$ \sum_{s\le t}\E [ Q^2_{s,t}(\E_{[u,v]}X_v, s< u \le v \le t)] \ < \ \infty $$
and (\ref{proSE}) for any $t \in \Z$.
\end{definition}

\begin{proposition} \label{proseGen} Assume that that $\mu_t = \mu$ does not depend on $t \in \R,$
 the functions $Q_{s,t} = Q_{t-s}, \, s\le t$ in (\ref{proSE}) depend only on $t-s,$ and that
$\{ X_t \}$ is a solution of (\ref{proSE}).
Then $\{ X_t \}$ is a projective moving average of (\ref{bernoulli}) with $\E X_t = \mu$ and
$g_n: \R^n \to \R, \, n = 0,1, \dots$ defined recursively by
\begin{equation}\label{grecurs}
g_{n}(x_{-n+1}, \dots, x_0) := Q_n\Big(\mu+ \sum_{k=u}^v x_k g_{v-k}(x_{u+1}, \dots, x_v), -n< u \le v \le 0\Big), \quad n \ge 1.
\end{equation}
$g_0 := Q_0.$ Moreover, such a solution is unique.
\end{proposition}

\begin{proof}
From (\ref{proSE}) and the uniqueness of (\ref{XMA}) (Proposition \ref{propunique}) we have
$X_t = \mu+ \sum_{s\le t} g_{s,t} \zeta_s, $ where $g_{s,t} = Q_{t-s}(\E_{[u,v]} X_v,  s < u \le v \le t).$
For $s = t$ this yields $g_{t,t} = Q_0 = g_0 \, \forall t \in \Z$ as in (\ref{grecurs}). Similarly,
$g_{t-1,t} = Q_{1}(\E_{[t,t]} X_t) =  Q_{1}(\mu+ g_0 \zeta_t) = g_1(\zeta_t)$, where $g_1$ is defined in
(\ref{grecurs}). Assume by induction that
\begin{equation}\label{ginduct}
g_{t-m, t} = g_m(\zeta_{t-m+1}, \dots, \zeta_t),  \qquad \forall \  t \in \Z
\end{equation}
with $g_m$ defined in (\ref{grecurs}),
hold for any $0\le  m < n$ and some
$n \ge 1$; we need to show that (\ref{ginduct}) holds  for $m=n$, too.  Using (\ref{ginduct}), (\ref{Eg}) and
(\ref{grecurs}) we
obtain
\begin{eqnarray*}
g_{t-n,t}&=&Q_{n}(\E_{[u,v]} X_v,  t-n < u \le v \le t) \\
&=&Q_n\Big(\mu+\sum_{k=u}^v \zeta_k g_{v-k}(\zeta_{u+1}, \dots, \zeta_v), t-n< u \le v \le t\Big) \\
&=&g_n(\zeta_{t-n+1}, \dots,  \zeta_t).
\end{eqnarray*}
This proves the induction step $n-1 \to n$ and hence the proposition, too, since
the uniqueness follows trivially. \hfill $\Box$
\end{proof}

\smallskip

Clearly, the choice of possible kernels $Q_{s,t}$ in (\ref{proSE}) is very large. In this paper
we focus on the following class of projective stochastic equations:
\begin{eqnarray}\label{proSEI}
X_t &=&\mu+ \sum_{s\le t} \zeta_s Q\Big(\alpha_{t-s}+   \sum_{s<u\le t}  \beta_{t-u, u-s} \left(\E_{[u,t]} X_{t} - \E_{[u+1,t]} X_{t}\right) \Big),
\end{eqnarray}
where $\{\alpha_i, i \ge 0\}, \ \{\beta_{i,j},  i \ge 0, j\ge 1 \} $ are given arrays of real numbers, $\mu \in \R$ is a constant,
and $Q = Q(x)$ is a measurable function
of a single variable $x \in \R$.  Two modifications of (\ref{proSEI}) are briefly discussed below, see
(\ref{proSEII}) and (\ref{proSE0}).
Particular cases of (\ref{proSEI}) are
\begin{eqnarray}\label{proSE1}
X_t &=&\sum_{s\le t} \zeta_s Q\Big(\alpha_{t-s}+   \beta_{t-s} \E_{[s+1,t]} X_t \Big),
\end{eqnarray}
and
\begin{eqnarray}\label{proSE2}
X_t &=&\mu + \sum_{s\le t} \zeta_s Q\Big(\alpha_{t-s}+ \sum_{s<u\le t}  \beta_{u-s}
\left(\E_{[u,t]} X_t - \E_{[u+1,t]} X_t\right) \Big),
\end{eqnarray}
corresponding to  $\beta_{i,j} = \beta_{i+j}$  and
$\beta_{i,j} = \beta_{j},$ respectively.

\smallskip

Next, we study the solvability of projective equation (\ref{proSEI}).
We assume that $Q$
satisfies the following dominating bound: there exists a constant $c_Q >0$ such that
\begin{eqnarray}
|Q(x)|&\le&c_Q|x|, \qquad \forall \ x \in {\R}. \label{Qdom}
\end{eqnarray}
Denote
\begin{eqnarray}
K_Q
&:=&\sum_{i=0}^\infty \alpha_i^2 \sum_{k=0}^\infty c^{2k+2}_Q   \sum_{j_1=1}^\infty \beta^2_{i,j_1} \cdots
\sum_{j_k=1}^\infty \beta^2_{i+ j_1 + \dots + j_{k-1}, j_k}. \label{KQ}
\end{eqnarray}
The main result of this section is the following theorem.

\begin{theorem} \label{thm1} (i) Assume condition (\ref{Qdom}) and
\begin{equation}\label{Kcond}
K_{Q} \ < \ \infty.
\end{equation}
Then there exists a unique solution $\{ X_t \} $ of (\ref{proSEI}), which is  written as
a projective moving average in (\ref{XMA})
with coefficients $g_{t-k,t}$ recursively defined as
\begin{eqnarray}\label{giter}
g_{t-k,t}&:=&\begin{cases}
Q\big(\alpha_k+ \sum_{i=0}^{k-1} \beta_{i,k-i} \zeta_{t-i} g_{t-i,t}  \big), &k =1,2, \dots, \\
Q(\alpha_k), &k=0.
\end{cases}
\end{eqnarray}
The above solution is represented by the following nested Volterra series
\begin{eqnarray*}
\hskip-.5cm X_t&=&\mu + \sum_{s_1 \le  t} \zeta_{s_1} G_{s_1} \Big( \sum_{s_1< s_2 \le t} \zeta_{s_2} G_{s_1,s_2}
\Big( \dots
\zeta_{s_{p-1}} G_{s_1,\dots,s_{p-1}}\Big(\sum_{s_{p-1} < s_p \le t} \zeta_{s_p} G_{s_1, \dots, s_p}
\Big) \Big) \Big),
\end{eqnarray*}
where
\begin{eqnarray*}
G_S(x)&=&
\begin{cases}
Q(\alpha_0), &S = \{t \}, \\
Q(\alpha_{t-s} + x), &S = \{s \},  \  s <t,  \\
\beta_{t-s_{k}, s_{k}-s_{k-1}}Q (\alpha_{t-s_k}),   &S = \{s_1, \dots ,s_k \}, \ s_1 < \dots < s_k = t,\, k>1, \\
\beta_{t-s_{k}, s_{k}-s_{k-1}}Q (\alpha_{t-s_k}+x),   &S = \{s_1, \dots ,s_k \}, \ s_1 < \dots < s_k < t, \, k>1.
\end{cases}
\end{eqnarray*}
More explicitly,
\begin{eqnarray*}
X_t&=&\mu + Q(\alpha_0) \zeta_t + Q\big(\alpha_1 + \beta_{0,1} \zeta_t Q(\alpha_{0})\big)\zeta_{t-1} \\
&+&Q\Big(\alpha_2 + \beta_{0,2}\zeta_t Q(\alpha_{0}) + \beta_{1,1}\zeta_{t-1} Q\big(\alpha_1 +\beta_{0,1} \zeta_t Q(\alpha_{0}) \big)  \Big)\zeta_{t-2} + \dots.
\end{eqnarray*}

\noi (ii) In the case of linear function $Q(x) = c_Q x$, condition (\ref{Kcond}) is also necessary
for the existence of a solution of (\ref{proSEI}).

\end{theorem}

\begin{proof}
(i) Let us show
that 
the $g_{k-t,t}$'s as defined in (\ref{giter}) satisfy
$\sum_{k=0}^\infty \E g^2_{t-k,t} < \infty$. From (\ref{Qdom}) and (\ref{giter})
we have the recurrent inequality:
\begin{eqnarray}
\hskip-.3cm \E g^2_{t-k,t}
&\le&c^2_Q \E \Big(\alpha_k+ \sum_{i=0}^{k-1} \beta_{i,k-i} \zeta_{t-i} g_{t-i,t}  \Big)^2 \
=\ c^2_Q \Big(\alpha^2_k+ \sum_{i=0}^{k-1} \beta^2_{i,k-i} \E g^2_{t-i,t}  \Big).
\label{gineq}
\end{eqnarray}
Iterating (\ref{gineq}) we obtain
\begin{eqnarray}
\hskip-.4cm \E g^2_{t-k,t}
&\le&c^2_Q \Big(\alpha^2_k+  c^2_Q\sum_{i=0}^{k-1} \beta^2_{i,k-i} \Big(\alpha^2_i + \sum_{j=0}^{i-1}
\beta^2_{j, i-j} \E g^2_{t-j,t} \Big) \Big) \nn \\
&=&c^2_Q \alpha^2_k + c^4_Q \sum_{i=0}^{k-1} \alpha^2_i \beta^2_{i,k-i} + c^6_Q \sum_{i=0}^{k-1} \alpha^2_i
\sum_{j_1=1}^{k-1-i} \beta^2_{i,j_1} \beta^2_{i+j_1, k-i-j_1} + \dots \label{gkiter}
\end{eqnarray}
and hence
\begin{eqnarray}
\sum_{k=0}^\infty \E g^2_{t-k,t}
&\le&c^2_Q \sum_{i=0}^\infty \alpha^2_i + c^4_Q \sum_{i=0}^\infty \alpha^2_i \sum_{j_1=1}^\infty \beta^2_{i,j_1}
+ c^6_Q \sum_{i=0}^\infty \alpha^2_i \sum_{j_1=1}^\infty \beta^2_{i,j_1} \sum_{j_2=1}^\infty \beta^2_{i+j_1, j_2} + \dots \nn \\
&=&K_{Q} \ < \ \infty \label{Kineq}
\end{eqnarray}
according to (\ref{Kcond}). Therefore, $X_t = \mu + \sum_{s\le t} g_{s,t} \zeta_s$   is a well-defined projective moving-average.
The remaining statements about $X_t$ follow from Proposition \ref{proseGen}.

\smallskip

\noi (ii) Similarly to (\ref{gineq}), (\ref{Kineq})  in the case $Q(x) = c_Q x$ we obtain
\begin{eqnarray*}
\E g^2_{t-k,t}
&=&c^2_Q \E \Big(\alpha_k+ \sum_{i=0}^{k-1} \beta_{i,k-i} \zeta_{t-i} g_{t-i,t}  \Big)^2 \
=\  c^2_Q \Big(\alpha^2_k+ \sum_{i=0}^{k-1} \beta^2_{i,k-i} \E g^2_{t-i,t}  \Big)
\end{eqnarray*}
and hence ${\rm Var}(X_t) = \sum_{k=0}^\infty \E g^2_{t-k,t} = K_Q $. This proves  (ii) and
Theorem \ref{thm1}, too. \hfill $\Box$
\end{proof}

\smallskip

\smallskip

\smallskip

In the case of equations (\ref{proSE1}) and (\ref{proSE2}),
condition (\ref{Kcond}) can be simplified, see below. Note that for $ A^2 := \sum_{i=0}^\infty \alpha^2_i = 0$,
equations (\ref{giter}) admit a trivial solution $g_{t-k,t} = 0 $ since $Q(0) = 0$ by (\ref{Qdom}), leading
to the constant process $X = \mu $ in (\ref{proSEI}).

\begin{proposition}\label{propBeta}
(i) Let $A^2 >0, $
$\beta_{i,j} = \beta_{i+j}, \, i\ge 0, \ j \ge 1, $
and $B^2 := \sum_{i=0}^\infty \beta^2_i$. Then
$K_Q<\infty $ is equivalent to $A^2 < \infty $ and $ B^2 < \infty $.

\smallskip

\noi (ii) Let $A^2 >0$,
$\beta_{i,j} = \beta_{j}, \, i\ge 0, \ j \ge 1$
and $B^2 := \sum_{i=1}^\infty \beta^2_i$.
Then
$K_Q <\infty $ is equivalent to $A^2 < \infty $ and $c^2_Q B^2 < 1 $. Moreover,
$K_Q = c^2_Q A^2/(1 - c^2_Q B^2).$

\end{proposition}

\begin{proof}
(i) By definition,
\begin{eqnarray*}
K_{Q}
&=&\sum_{k=0}^\infty c^{2k+2}_Q \sum_{i=0}^\infty \alpha_i^2 \sum_{j_1=1}^\infty \beta^2_{i+j_1} \cdots
\sum_{j_k=1}^\infty \beta^2_{i+ j_1 + \dots + j_{k-1}+ j_k} \\
&=&\sum_{k=0}^\infty c^{2k+2}_Q \sum_{0\le i< j_1 < \dots < j_k< \infty} \alpha_i^2 \beta^2_{j_1} \cdots
\beta^2_{j_k} \nn \\
&\le&\sum_{k=0}^\infty c^{2k+2}_Q A^2 B^2_1 \cdots B^2_k, \nn
\end{eqnarray*}
where $B^2_k := \sum_{j=k}^\infty \beta^2_j.$ Since $B^2 < \infty $ entails $\lim_{k\to \infty} B^2_k = 0$,
$\forall \epsilon>0 \, \exists \, K \ge 1$ such that
$B^2_k < \epsilon/ c^2_Q $ $ \forall \, k  > K$. Hence,
\begin{eqnarray*}
K_{Q}&\le&c^2_Q A^2\Big(\sum_{k=0}^K (c^{2}_Q  B^2)^k + \sum_{k=K}^\infty \epsilon^k \Big) \ < \infty.
\end{eqnarray*}
Therefore, $A^2 < \infty $ and $B^2 < \infty $ imply $K_{Q} < \infty$. The converse implication is
obvious.

\smallskip

\noi (ii) Follows by
\begin{eqnarray*}
K_{Q}
&=&\sum_{k=0}^\infty c^{2k+2}_Q \sum_{i=0}^\infty \alpha_i^2 \sum_{j_1=1}^\infty \beta^2_{j_1} \cdots
\sum_{j_k=1}^\infty \beta^2_{j_k} \
=\  \sum_{k=0}^\infty c^{2k+2}_Q A^2 (B^2)^k \ = \  \frac{c^2_Q A^2}{1 - c^2_Q B^2}.
\end{eqnarray*}
\end{proof}

\smallskip

\begin{remark} \label{rem5} {\rm It is not difficult to show that conditions on the $\beta_{i,j}$'s
in Proposition \ref{propBeta} (i) and (ii) are part of the following more general
condition: $\limsup_{i\to \infty} \sum_{j=1}^\infty c_Q^2 \beta^2_{i,j} < 1, $ which also guarantees that
$K_Q < \infty $.}
\end{remark}

\smallskip

The following Proposition \ref{propmom1} obtains a sufficient condition for the existence
of higher moments $\E |X_t|^p < \infty, p > 2$ of the solution of projective
equation (\ref{proSEI}). The proof of Proposition \ref{propmom1} is based on a recurrent
use of Rosenthal-type inequality of Proposition \ref{propmom}, which contains an absolute constant
$C_p$ depending only on $p$. For $p\ge 2$, denote
\begin{eqnarray}
K_{Q, p}
&:=&C_p^{2/p}  \sum_{i=0}^\infty \alpha_i^2 \sum_{k=0}^\infty (c_Q C^{1/p}_p \mu^{1/p}_p)^{2k+2}   \sum_{j_1=1}^\infty  \beta^2_{i,j_1} \cdots
\sum_{j_k=1}^\infty \beta^2_{i+ j_1 + \dots + j_{k-1}, j_k}. \label{KQp}
\end{eqnarray}
where (recall) $\mu_p = \E |\zeta_0|^p$. Note
$C_2 = \mu_2 = 1 $, hence  $K_{Q,2} = K_Q $ coincides with (\ref{KQ}).

\begin{proposition} \label{propmom1} Assume conditions of Theorem~\ref{thm1} and
$K_{Q, p}<\infty, $ for some $p\ge 2$.
Then $\E |X_t|^p < \infty $.
\end{proposition}

\begin{proof}
The proof is similar to that of Theorem~\ref{thm1} (i).
By Proposition \ref{propmom},
\begin{eqnarray*}
\big(\E \big|X_t\big|^p\big)^{2/p} & \le & C_p^{2/p} \Big( \big|\E X_t\big|^p + \mu_p \big(\sum_{s \le t} \big(\E |g_{s,t}|^p \big)^{2/p} \big)^{p/2}\Big)^{2/p}\ =\  C_p^{2/p}  \mu_p^{2/p}\sum_{s \le t} (\E |g_{s,t}|^p )^{2/p}.
\end{eqnarray*}
Using condition \eqref{Qdom}, Proposition~\ref{propmom} and inequality $(a+b)^q \le a^q+b^q, \, 0<q\le1$ we obtain the following recurrent inequality:
\begin{eqnarray*}
\big(\E |g_{s,t}|^p\big)^{2/p} & \leq & \big(c_Q^p\E \big|\alpha_{t-s}+   \sum\nolimits_{s<u\le t}  \beta_{t-u, u-s} \zeta_u g_{u,t} \big|^p \big)^{2/p} \\
& \leq & c_Q^2 C_p^{2/p} \Big(|\alpha_{t-s}|^p+  \mu_p \big( \sum\nolimits_{s<u\le t} ( |\beta_{t-u, u-s}|^p\,\E |g_{u,t}|^p)^{2/p} \big)^{p/2} \Big)^{2/p}. \\
& \leq & c_Q^2 C_p^{2/p} \Big(|\alpha_{t-s}|^2+  \mu_p^{2/p} \sum_{s<u\le t} \beta_{t-u, u-s}^2 (\E |g_{u,t}|^p)^{2/p}
\Big).
\end{eqnarray*}
Iterating the last inequality as in the proof of Theorem \ref{thm1} we obtain
$(\E |X_t|^p)^{2/p} \le {K}_{Q,p}<\infty$, with $K_{Q,p} $ given in (\ref{KQp}). \hfill $\Box$
\end{proof}

\smallskip

Finally, let us discuss the question when $X_t$ of (\ref{proSEI}) satisfies the weak
dependence condition in (\ref{DedMer}) for the invariance principle.

\begin{proposition} Let $\{X_t \}$ satisfy the conditions of Theorem~\ref{thm1}  and $\Omega(2)$ be defined
in (\ref{DedMer}). Then
\begin{eqnarray}
\Omega(2)
&\le&\sum_{i=0}^\infty |\alpha_i| \sum_{k=0}^\infty c^{k+1}_Q   \sum_{j_1=1}^\infty |\beta_{i,j_1}| \cdots
\sum_{j_k=1}^\infty |\beta_{i+ j_1 + \dots + j_{k-1}, j_k}|. \label{KKQ}
\end{eqnarray}
In particular, if the quantity on the r.h.s. of (\ref{KKQ}) is finite, $\{X_t \}$ satisfies
the functional central limit theorem in (\ref{DedMer1}).

\end{proposition}

\begin{proof} Follows from (\ref{gkiter}) and the inequality $|\sum x_i |^{1/2} \le \sum |x_i|^{1/2}$.
\hfill $\Box$

\end{proof}

\section{Examples}

\noi {\it 1. Finitely dependent projective equations.} Consider equation  \eqref{proSEI}, where
$\alpha_i = \beta_{i,j} = 0 $ for all $i > m $ and some $m \ge 0$. Since $Q(0) = 0$, the corresponding equation writes as
\begin{eqnarray}\label{proSEI0}
X_t &=&\mu+ \sum_{t-m< s\le t} \zeta_s Q\Big(\alpha_{t-s}+   \sum_{s<u\le t}  \beta_{t-u, u-s} \left(\E_{[u,t]} X_{t} - \E_{[u+1,t]} X_{t}\right) \Big),
\end{eqnarray}
where the r.h.s. is ${\cal F}_{[t-m+1,t]}$-measurable. In particular, $\{X_t\}$ of \eqref{proSEI0} is an $m$-dependent process.
We may ask if the above process can be represented as a moving-average of length $m$ w.r.t. to some i.i.d. innovations?
In other words,
if there exists an i.i.d. standardized sequence
$\{\eta_s, s\in \Z\}$ and coefficients $c_j, 0\le j < m$ such that
\begin{equation}\label{Xma}
X_t = \sum_{t-m < s \le t} c_{t-s} \eta_s, \qquad t \in \Z.
\end{equation}

To construct a negative counter-example to the above question,
consider the simple  case of  \eqref{proSEI0} with $m=2$, $\mu = 0$,  $\alpha_1 = 0, \beta_{0,1} = 1, Q(\alpha_0) = 1$:
\begin{eqnarray}\label{proSEI00}
X_t&=&\zeta_t Q(\alpha_0) + \zeta_{t-1} Q(\alpha_1 + \beta_{0,1} \E_{[t,t]} X_t) =  \zeta_t + \zeta_{t-1} Q( \zeta_t).
\end{eqnarray}
Assume that $\E Q(\zeta_t)= 0$. Then $\E X_t X_{t-1} = 0, \E X^2_t = 1 + \E Q^2(\zeta_0)$. On the other hand, from
\eqref{Xma}  with $m=2 $
we obtain $0= \E X_t X_{t-1} = c_0 c_1, $  implying that
$\{X_t \}$ is an i.i.d. sequence.

Let us show that the last conclusion contradicts the form of $X_t$ in    \eqref{proSEI00} under general assumptions on $Q$ and the distribution of
$\zeta = \zeta_0$. Assume
that $\zeta$ is symmetric, $ \infty >  \E \zeta^4 > (\E \zeta^2)^2 = 1 $ and $Q$ is antisymmetric. Then
\begin{eqnarray*}
{\rm cov}(X^2_t,X^2_{t-1})
&=&\E Q^2(\zeta)\big\{(\E \zeta^4 -  1) + (\E \zeta^2 Q^2(\zeta) - \E Q^2(\zeta))\big\}.
\end{eqnarray*}
Assume, in addition, that $Q$ is monotone nondecreasing on $[0, \infty)$. Then  $\E \zeta^2 Q^2(\zeta) \ge \E \zeta^2 \E Q^2(\zeta) = \E Q^2(\zeta)$, implying
${\rm cov}(X^2_t,X^2_{t-1}) >0$. As a consequence,
\eqref{proSEI00}  is not a moving average of length 2 in some standardized i.i.d. sequence.

\medskip

\noi {\it 2. Linear kernel $Q$.} For linear kernel $Q(x) = c_Q x $, the solution of (\ref{proSEI}) of Theorem \ref{thm1}
can be written explicitly as $X_t = \mu+ \sum_{k=1}^\infty X^{(k)}_t$, where $X^{(1)}_t = c_Q\sum_{i=0}^\infty \alpha_i \zeta_{t-i} $ is a linear process and \begin{eqnarray*}
X^{(k+1)}_t&=&c_Q^{k+1}\sum_{i=0}^\infty \alpha_i \sum_{j_1, \dots, j_k =1}^\infty
\beta_{i,j_1} \cdots \beta_{i+j_1+ \dots + j_{k-1}, j_k} \zeta_{t-i} \zeta_{t-i-j_1} \cdots \zeta_{t-j_1 - \dots -  j_k}
\end{eqnarray*}
for $k\ge 1$ is a Volterra series of order $k+1$, see (\cite{ref5}, p.22),
which are orthogonal in sense that $\E X^{(k)}_t X^{(\ell)}_s =0, \, t,s \in \Z, \, k,\ell \ge 1,\, k \ne \ell$.

Let $H^2_{(-\infty,t]} \subset L^2_{(-\infty,t]} $ be the subspace spanned by
products $1, \zeta_{s_1} \cdots \zeta_{s_k}, \, s_1 <  \dots < s_k \le t, k\ge 1.$  Clearly, the above Volterra series $X_t, X^{(k)}_t \in H^2_{(-\infty, t]}, \, \forall  t\in \Z $ (corresponding to linear $Q$) 
constitute a very special class of projective processes. For example, the process in
\eqref{proSEI00} cannot be expanded in such series unless $Q$ is linear.  To show the last fact, decompose \eqref{proSEI00} as
$X_t = Y_t + Z_t$, where
$Y_t :=  \zeta_t + \alpha \zeta_{t-1} \zeta_t \in H^2_{(-\infty,t]}, \, \alpha :=  \E \zeta Q(\alpha) $ and
$Z_t := \zeta_{t-1} (Q(\zeta_t)  - \alpha \zeta_t)$ is orthogonal to  $H^2_{(-\infty,t]}, \, Z_t \neq 0,$ 
hence $X_t \not\in H^2_{(-\infty,t]}$.

\medskip

\noi {\it 3. The LARCH model.}
The Linear ARCH (LARCH) model, introduced by Robinson \cite{ref15}, is defined by the equations
\begin{equation}
\label{larch} r_t = \sigma_t \zeta_t, \ \ \ \sigma_t= \alpha + \sum_{j= 1}^\infty \beta_j r_{t-j},
\end{equation}
where $\{\zeta_t\}$ is a standardized i.i.d. sequence, and the coefficients $\beta_j$ satisfy $B := \Big\{\sum_{j=1}^\infty
\beta^2_j \Big\}^{1/2} < \infty $. The LARCH model was studied in \cite{ref7}, \cite{ref8}, \cite{ref10}), \cite{ref9},
\cite{ref3} and other papers.
In financial modeling, $r_t$ are interpreted as (asset) returns and $\sigma_t$ as  volatilities.
Of particular interest is the case when the $\beta_j$'s in (\ref{larch})
are proportional to
ARFIMA coefficients, in which case it is possible to rigorously prove long memory of the volatility and the
(squared) returns.
It is well-known (\cite{ref9})
that a second order strictly stationary solution $\{r_t\}$ to (\ref{larch}) exists if and only if
\begin{equation} \label{B2}
B \ < \ 1,
\end{equation}
in which case it can be represented by the convergent orthogonal Volterra series
\begin{equation*} \label{orthog}
r_t = \sigma_t\zeta_t, \ \ \ \sigma_t= \alpha\Big(1+\sum_{k= 1}^\infty \sum_{j_1, \dots, j_k= 1}^\infty \beta_{j_1}\cdots\beta_{j_k}
\zeta_{t-j_1}\cdots\zeta_{t-j_1-\dots-j_k}\Big).
\end{equation*}
Clearly, the last series is a particular case of the Volterra series of the previous example. 
We conclude  that under the condition (\ref{B2}),
the volatility process $\{X_t = \sigma_t\}$ of the LARCH model satisfies the projective
equation (\ref{proSE2}) with linear function $Q(x) = x $ and $\alpha_j =  \alpha \beta_j$. Note that
(\ref{B2}) coincides with the condition $c^2_Q B^2 < 1 $ of Proposition \ref{propBeta} (ii) for the existence
of solution of (\ref{proSE2}).

From Proposition  \ref{propmom1} the following new result about the existence of
higher order moments of the LARCH model is derived.

\begin{corollary} Assume that
\begin{equation}\label{larchp}
C_p^{1/p} \mu^{1/p}_p B <  1,
\end{equation}
where $\mu_p = \E |\zeta_0|^p $ and $C_p$ is the absolute constant from Proposition  \ref{propmom}, \,
$p \ge 2$. Then
$\E |r_t|^p  = \mu_p \E |\sigma_t|^p < \infty $. Moreover,
\begin{equation}\label{larchmom}
\E |\sigma_t|^p \ \le  \  \frac{\alpha^2 C_p^{4/p} \mu_p^{2/p} B^2}{1- C_p^{2/p} \mu_p^{2/p} B^2}.
\end{equation}
\end{corollary}

\begin{proof} Follows from Proposition  \ref{propmom1} and the easy fact that for the LARCH model,
$K_{Q,p}$ of (\ref{KQp}) coincides with the r.h.s. of
(\ref{larchmom}). \hfill $\Box$
\end{proof}

\smallskip

Condition (\ref{larchp}) can be compared with the sufficient condition for $\E |r_t|^p  < \infty, p = 2,4,\dots $
in (\cite{ref7}, Lemma 3.1):
\begin{equation}\label{larchp1}
(2^p - p -1)^{1/2} \mu^{1/p}_p B <  1.
\end{equation}
Although the best constant $C_p$ in the Rosenthal's inequality is not known,
(\ref{larchp}) seems much weaker than
(\ref{larchp1}), especially when $p$ is large. See, e.g. \cite{ref13}, where
it is shown that  $C^{1/p}_p = O(p/\log p),\, p \to \infty. $

\medskip

\noi {\it 4. Projective "threshold" equations.  } Consider projective equation
\begin{eqnarray}\label{proSET}
X_t &=&\zeta_t +  \sum_{j=1}^p \zeta_{t-j}  Q\big(\E_{[t-j+1,t]} X_{t} \big),
\end{eqnarray}
where $1\le p < \infty$ and
$Q $ is  a bounded measurable function with $Q(0) = 1$.  If $Q$ is a step function:
$Q(x) = \sum_{k=1}^q c_k \1(x \in I_k), $  where $\cup_{k=1}^q I_k = \R$ is a partition of
$\R$ into disjoint intervals $I_k, 1\le k \le q$, the process in \eqref{proSET} follows
different ``moving-average regimes'' in the regions $\E_{[t-j+1,t]} X_{t} \in I_k, 1\le j \le p $ exhibiting
a ``projective threshold effect''.  See Fig. \ref{QTR}, where the left graph shows a trajectory
having a single  threshold at $x=0$ and the right graph a trajectory with  two threshold points at $x=0$ and $x=2$.

\begin{figure}[!t]
	\begin{center}
		\subfigure{ \includegraphics[height=0.20\textwidth]{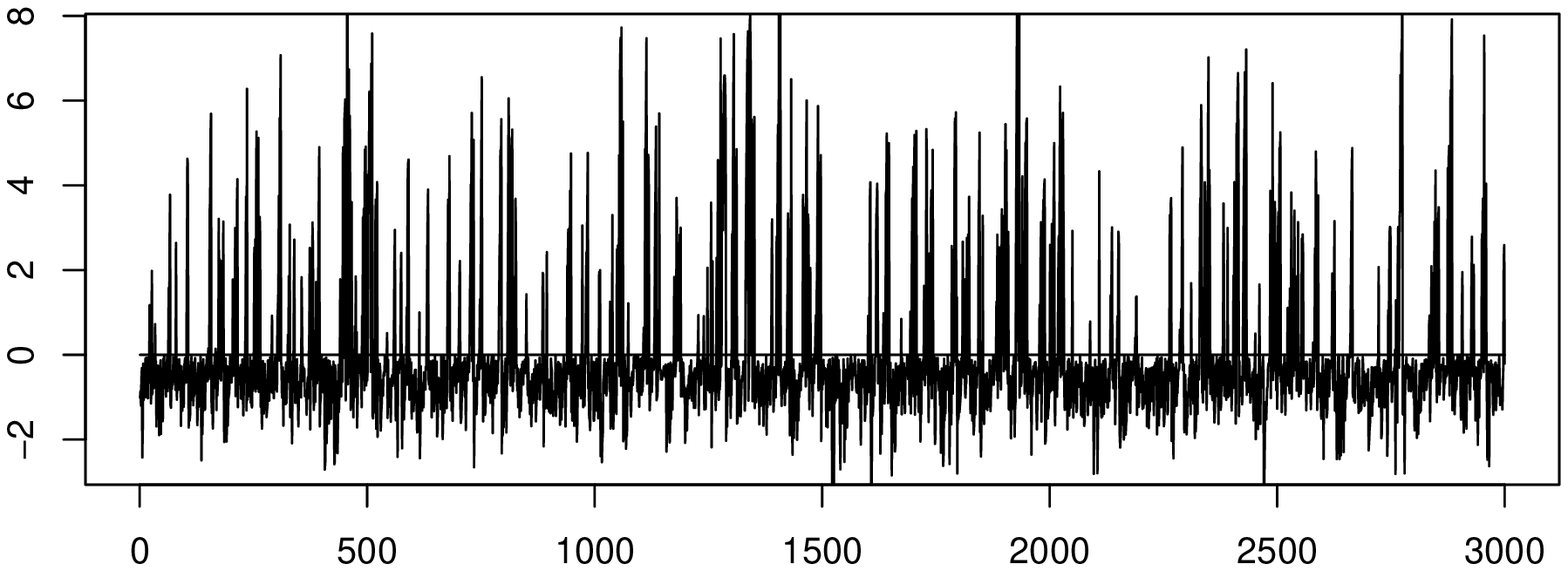}}
		\subfigure{\includegraphics[height=0.20\textwidth]{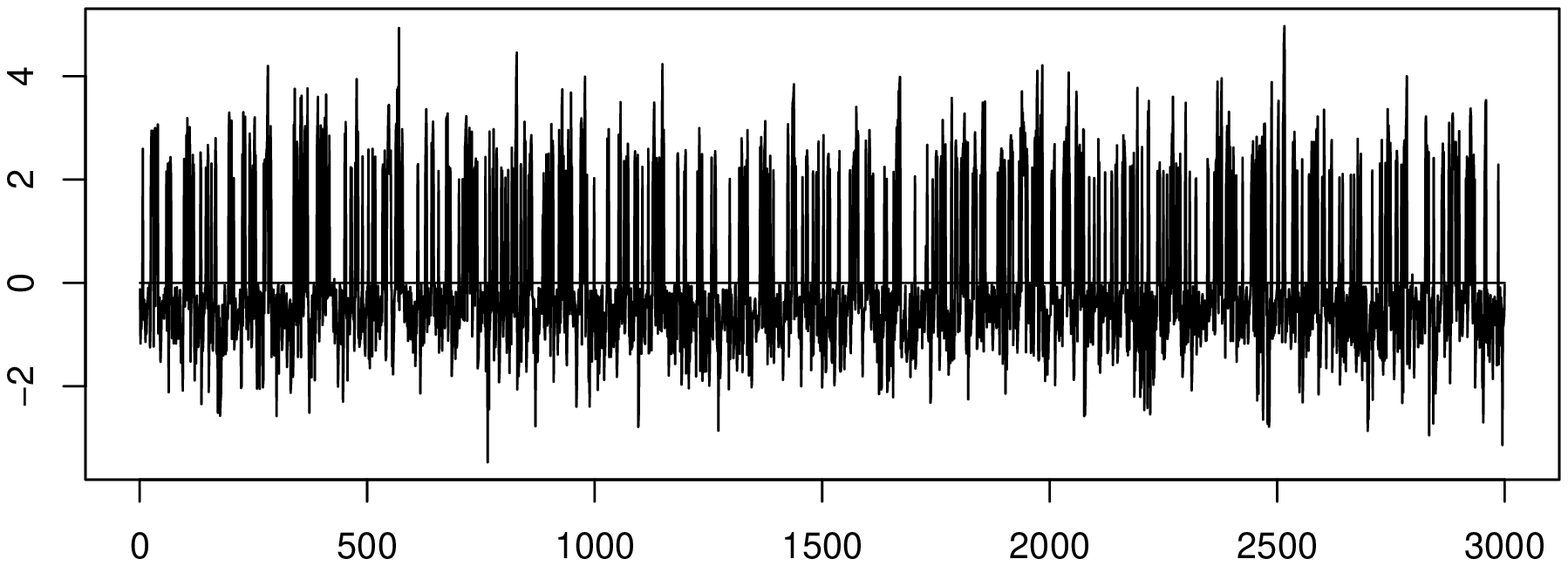}} 
	\end{center}
\caption{\small Trajectories  of solutions of  (\ref{proSET}), $p=10$. Left: $Q(x) = \1(x>0)$, right: $Q(x) = \1(0<x< 2)$. }
\label{QTR}
\end{figure}


\section{Simulations}

Solutions of projective equations can be
easily simulated using a truncated expansion $X^{(M)}_t = \sum_{t-M \le s \le t} g_{s,t} \zeta_s $
instead of infinite series in (\ref{proma0}). We chose the truncation level $M$ equal to the sample size $M= n = 3000 $ in
the subsequent simulations. The coefficients $g_{s,t} $  of projective equations
are computed very fast from recurrent formula (\ref{giter}) and simulated values
$\zeta_s, -M \le s \le n$.  The innovations were taken standard normal. For better comparisons,
we used the same sequence $\zeta_s, -M \le s \le n$ in all simulations.

Stationary solution of equation (\ref{proSE2}) was simulated for three different choices of $Q$ and
two choices of coefficients $\alpha_j, \beta_j$. The first choice of coefficients is $\alpha_j = 0.5^j, \beta_j= c\, 0.9^j $ and
corresponds to a short memory process $\{X_t\}$.
The second choice is $\alpha_j = \Gamma(d+ j)/\Gamma(d) \Gamma (j+1), \, \beta_j = c \alpha_j$ with $d = 0.4$
corresponds to a long memory process $\{X_t\}$ with coefficients as in ARFIMA$(0,d,0)$.
The value of $c>0$
was chosen so that $ c^2_Q B^2 = 0.9 < 1 $. The latter condition guarantees the existence
of a stationary solution of  (\ref{proSE2}), see Proposition \ref{propBeta}.

The simulated trajectories and (smoothed) histograms of marginal densities strongly depend on the kernel $Q$. We
used the following functions:
\begin{equation} \label{Qsimul}
Q_1(x) = x,  \qquad Q_2(x)=\max(0,x), \qquad Q_3(x) = \begin{cases}
x,&x\in [0, 1],\\
2-x,&x \in [1, 2],\\
0,&\text{otherwise}.
\end{cases}
\end{equation}
Clearly, $Q_i, i=1,2,3$ in (\ref{Qsimul}) satisfy (\ref{Qdom}) with $c_Q =1 $ and the Lipschitz condition
(\ref{QLip}). Note that $Q_3$ is bounded and supported in the compact interval $[0,2]$ while
$Q_1, Q_2 $ are unbounded, the latter being bounded from below. Also
note that for $\beta_j \equiv 0$ and the choice of $\alpha_j$ as above, the projective
process $\{X_t \}$ of (\ref{proSE2}) agrees with AR(.5) for $\alpha_j = 0.5^j$ and with ARFIMA$(0,0.4,0)$ for $\alpha_j = \Gamma(d+ j)/\Gamma(d) \Gamma (j+1)$
in all three cases  in (\ref{Qsimul})

\begin{figure}[!t]
	\begin{center}
		\subfigure{ \includegraphics[height=0.20\textwidth]{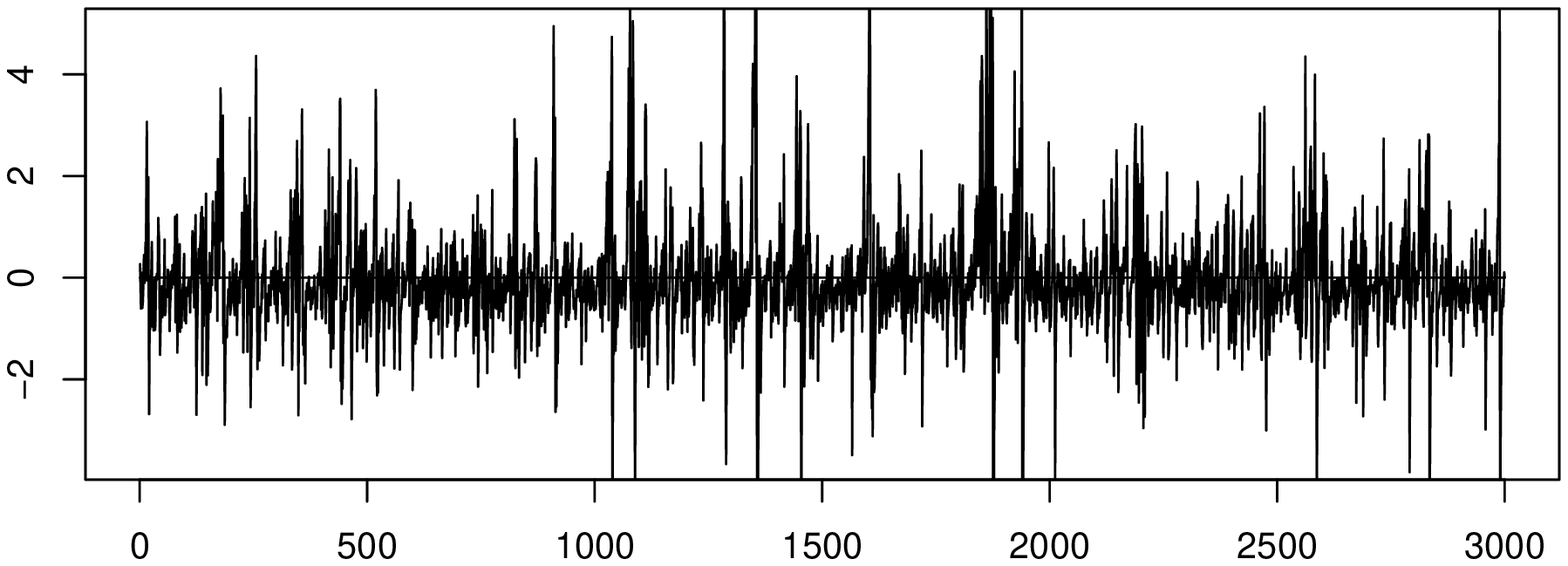}}
		\subfigure{\includegraphics[height=0.20\textwidth]{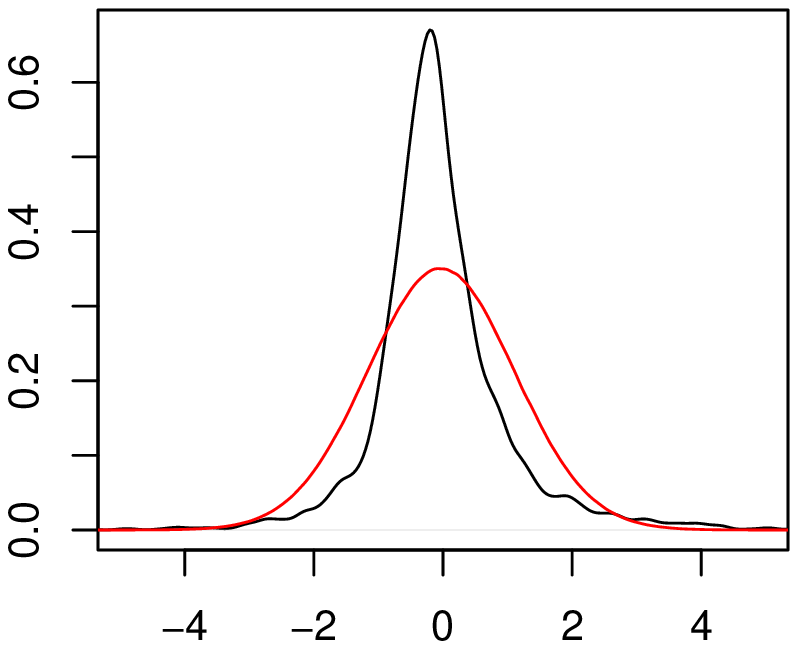}} \\
\subfigure{ \includegraphics[height=0.20\textwidth]{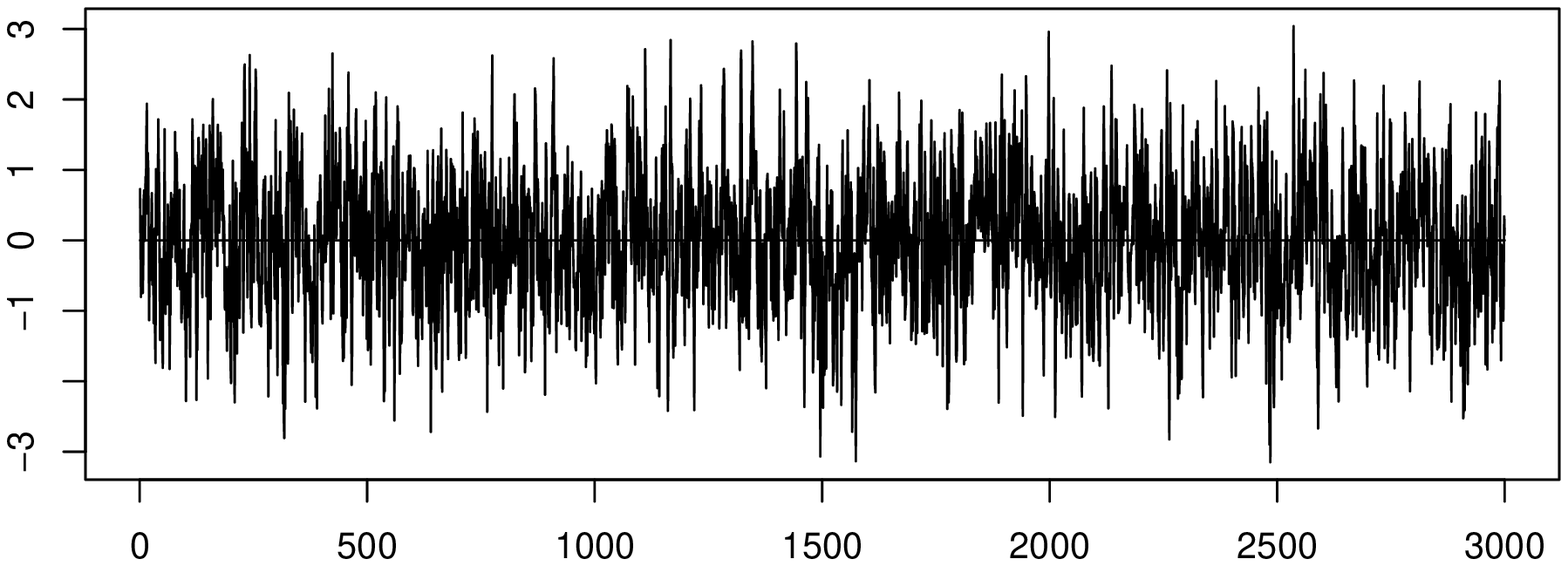}}
		\subfigure{\includegraphics[height=0.20\textwidth]{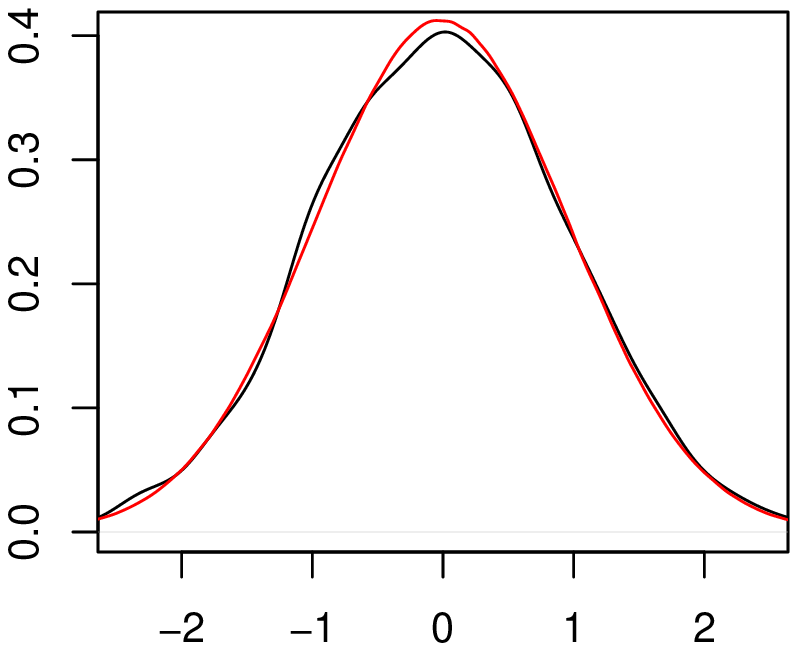}}
	\end{center}
\caption{\small Trajectories  and (smoothed) histograms of solutions of projective equation (\ref{proSE2}) with $Q(x) = Q_1(x) = x$. Top: $\alpha_j = (.5)^j, \beta_j = c (.9)^j$,
bottom: $\alpha_j = (.5)^j, \beta_j = 0$.}
\label{lmQint1}
\end{figure}

\begin{figure}[!t]
	\begin{center}
		\subfigure{ \includegraphics[height=0.20\textwidth]{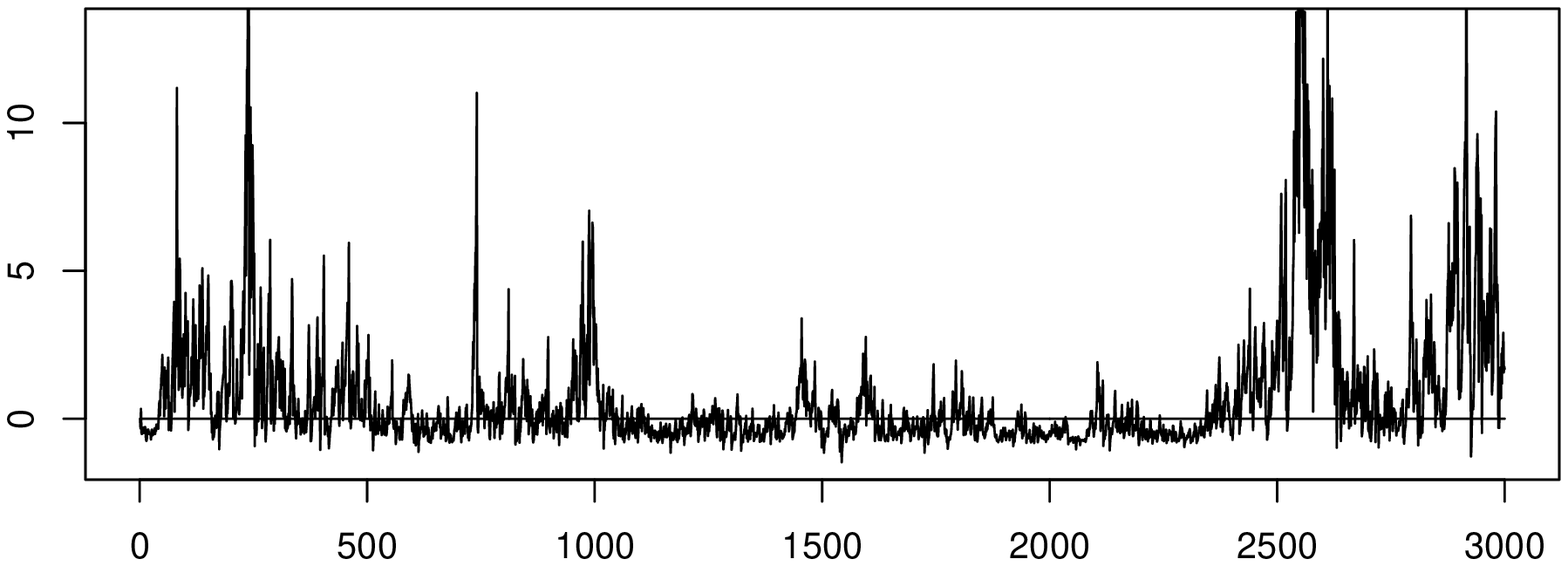}}
		\subfigure{\includegraphics[height=0.20\textwidth]{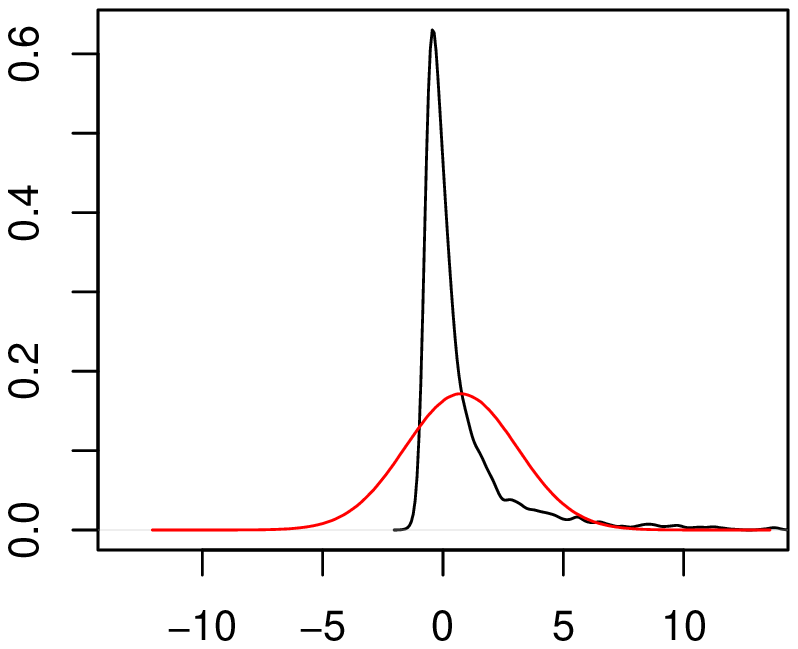}} \\
\subfigure{ \includegraphics[height=0.20\textwidth]{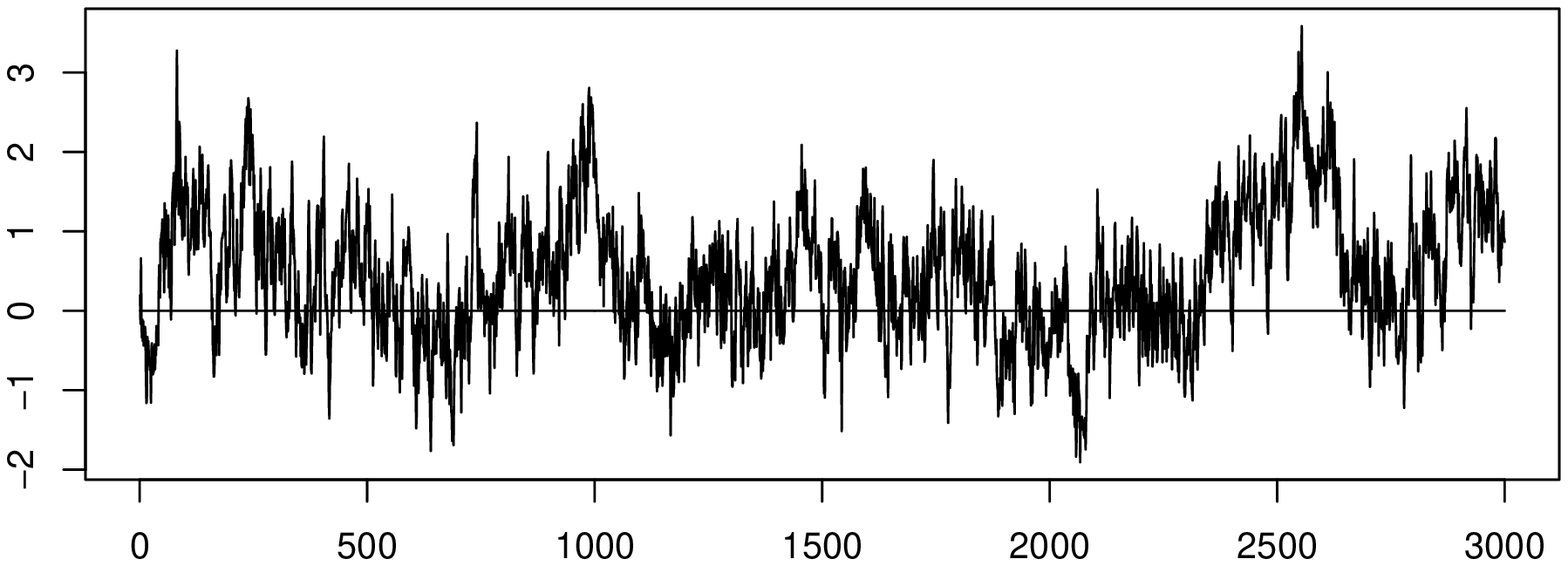}}
		\subfigure{\includegraphics[height=0.20\textwidth]{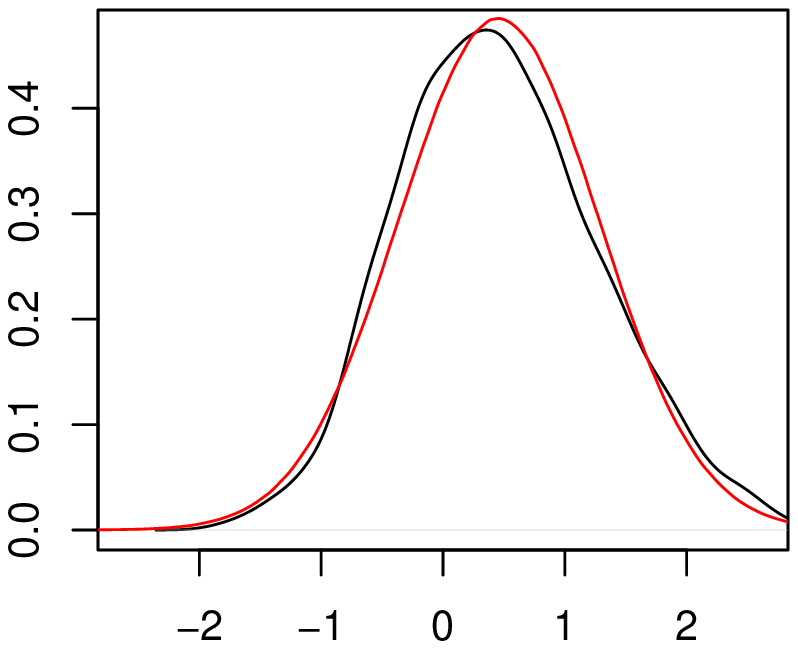}}
	\end{center}
\caption{\small Trajectories   and (smoothed) histograms of solutions of projective  equation (\ref{proSE2}) with $Q(x) = Q_2(x) = \max(x,0).$ Top: $\alpha_j = {\rm ARFIMA}(0,0.4,0), \beta_j = c \, \alpha_j$, bottom:
$\alpha_j = {\rm ARFIMA}(0,0.4,0), \beta_j = 0$.}
\label{lmQint2}
\end{figure}

\begin{figure}[!t]
	\begin{center}
		\subfigure{ \includegraphics[height=0.20\textwidth]{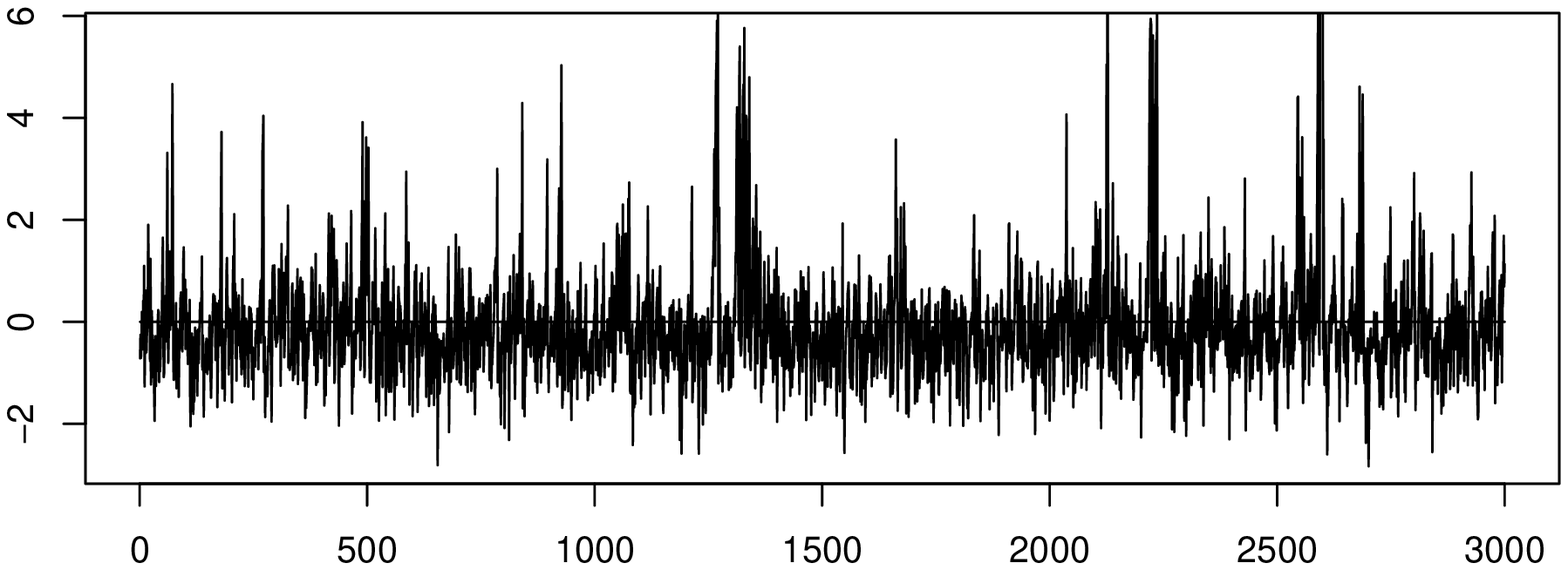}}
		\subfigure{\includegraphics[height=0.20\textwidth]{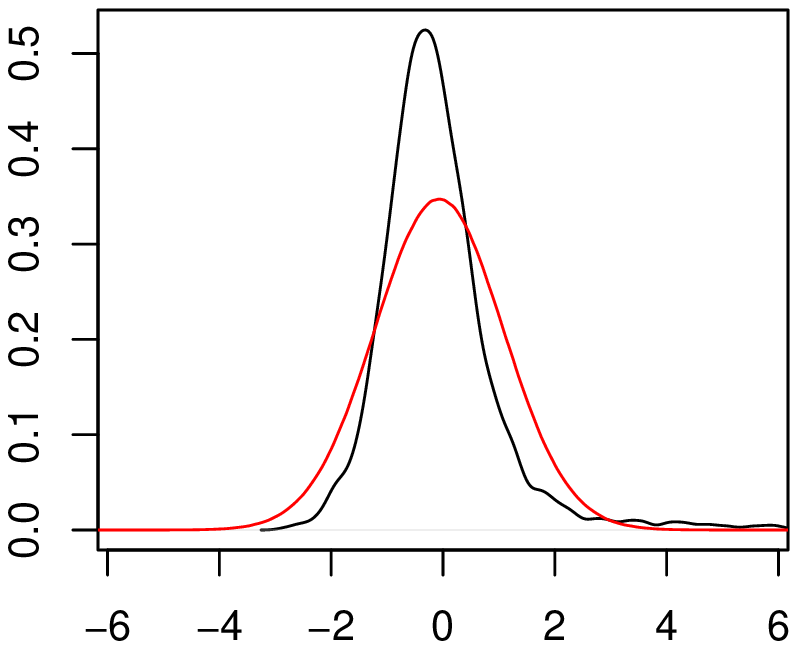}} \\
\subfigure{ \includegraphics[height=0.20\textwidth]{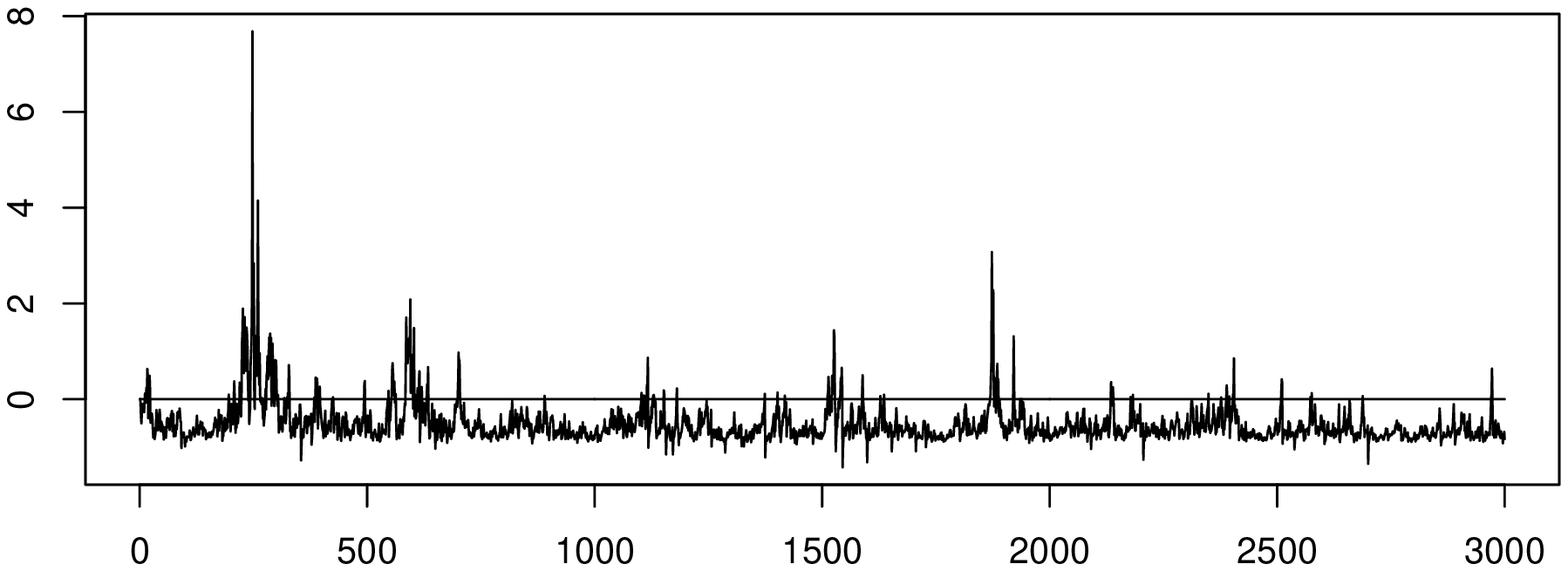}}
		\subfigure{\includegraphics[height=0.20\textwidth]{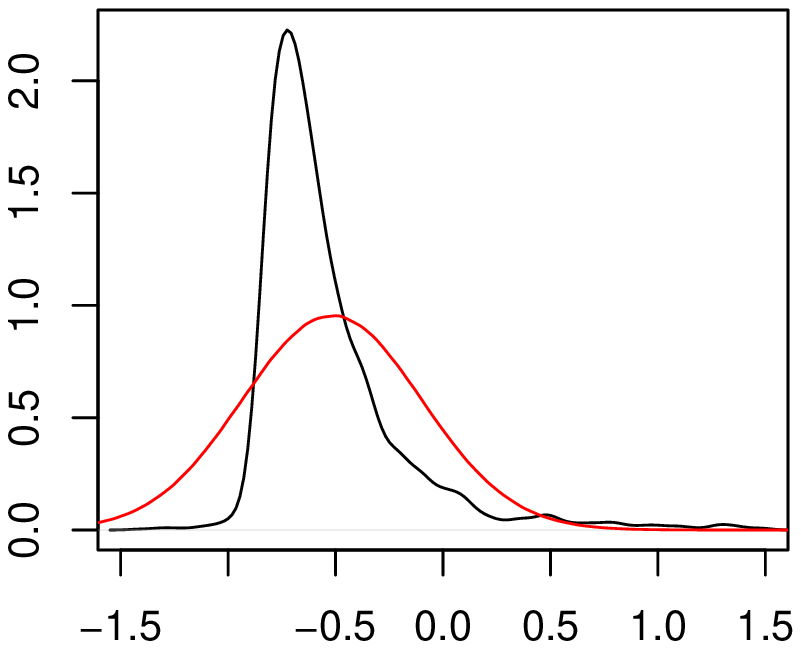}}
	\end{center}
\caption{\small Trajectories and (smoothed) histograms of solutions of projective equation (\ref{proSE2}) with $Q(x) = Q_3(x)$ = the ``triangle function'' in
(\ref{Qsimul}).
Top:  $\alpha_j = (.5)^j, \beta_j = c (.9)^j$, bottom: $\alpha_j = {\rm ARFIMA}(0,0.4,0), \beta_j = c \, \alpha_j$}
\label{lmQint3}
\end{figure}

A general impression from our simulations is that in all cases of $Q$ in (\ref{Qsimul}), the coefficients $\alpha_j$ account for the persistence
and  $\beta_{j}$ for the clustering of the process.  We observe that as $\beta_j$'s increase, the process becomes more asymmetric and its empirical
density diverges from the normal density (plotted in red in Figures~\ref{lmQint1}-\ref{lmQint3}
with parameters equal to the empirical mean and variance of the simulated series).
In the case of unbounded $Q = Q_1, Q_2$  and long memory ARFIMA coefficients, the marginal distribution seems
strongly skewed to the left and having a very light left tail and a much heavier right tail. On the other hand,
in the case of geometric coefficients, the density for $Q = Q_1, Q_2$ seems rather symmetric although heavy tailed. Case of $Q=Q_3$
corresponding to
bounded $Q$ seems to result in asymmetric distribution with light tails.

\section{Modifications}


Equation (\ref{proSEI}) can be modified in several ways. The first modification is
obtained by taking
the $\alpha_{t-s}$'s ``outside of $Q$'':
\begin{eqnarray}\label{proSEII}
X_t&=&\mu + \sum_{s\le t} \zeta_s \alpha_{t-s} Q\Big(\sum_{s<u\le t}  \beta_{t-u, u-s} \left(\E_{[u,t]} X_t - \E_{[u+1,t]} X_t\right) \Big),
\end{eqnarray}
where $\alpha_i, \beta_{i,j}, Q$ satisfy similar conditions as in (\ref{proSEI}). However, note that (\ref{Qdom}) implies
$Q(0) = 0$ in which case (\ref{proSEII}) has a trivial solution $X_t \equiv \mu $.  To avoid
the last eventuality, condition (\ref{Qdom}) must be changed.
Instead,
we shall assume  that $Q$ is a measurable function
satisfying
\begin{eqnarray}
Q(x)^2&\le&c^2_0 + c^2_1x^2,  \quad x \in {\R}  \label{QdomII}
\end{eqnarray}
for some $c_0, c_1 \ge 0$.
Denote
\begin{eqnarray*}
\tilde K_{Q}&:=&c^2_0
\sum_{k=0}^\infty c_1^{2k} \sum_{i=0}^\infty \alpha_i^2 \sum_{j_1=1}^\infty \alpha^2_{i+j_1}\beta^2_{i,j_1} \cdots
\sum_{j_k=1}^\infty \alpha^2_{i+ j_1 + \dots + j_k} \beta^2_{i+ j_1 + \dots + j_{k-1}, j_k}.
\end{eqnarray*}
Proposition \ref{thm2} can be proved similarly to Theorem \ref{thm1} and its proof is omitted.

\smallskip

\begin{proposition} \label{thm2} (i) Assume condition (\ref{QdomII}) and
\begin{equation}\label{Kcond2}
\tilde K_{Q} \ < \ \infty.
\end{equation}
Then there exists a unique solution $\{ X_t \} $ of (\ref{proSEII}), which is  written as
a projective moving average of (\ref{XMA})
with coefficients $g_{t-k,t}$ recursively defined as
\begin{eqnarray*}
g_{t-k,t}&:=&\alpha_k
Q\Big(\sum_{i=0}^{k-1} \beta_{i,k-i} \zeta_{t-i} g_{t-i,t}  \Big), \qquad k =1,2, \dots, \quad g_{t,t}:= \alpha_0 Q(0).
\label{giterII}
\end{eqnarray*}

\noi (ii) In the case of linear function $Q(x) = c_0 + c_1 x$, condition (\ref{Kcond2}) is also necessary
for the existence of a solution of (\ref{proSEII}).

\end{proposition}

\begin{remark}\label{remMod}
{\rm Let $A^2_k := \sum_{i=k}^\infty \alpha^2_i$ and $|\beta_{i,j}| \le \bar \beta $. Then
\begin{eqnarray*}
\tilde K_{Q}&\le&c^2_0
\sum_{k=0}^\infty (c_1 \bar \beta)^{2k} \sum_{i=0}^\infty \alpha_i^2   \sum_{j_1=1}^\infty \alpha^2_{i+j_1} \cdots
\sum_{j_k=1}^\infty \alpha^2_{i+ j_1 + \dots + j_k} \
\le\ c^2_0
\sum_{k=0}^\infty (c_1 \bar \beta)^{2k} A^2_0 A^2_1 \cdots A^2_k.
\end{eqnarray*}
Hence, $A^2 = A^2_0 < \infty $ and $\bar \beta < \infty $ imply $\tilde K_{Q}< \infty, $ for any
$c_0, c_1, \bar \beta $; see the proof of Proposition \ref{propBeta}.
}
\end{remark}

\smallskip

Projective stochastic equations (\ref{proSEI}) and (\ref{proSEII}) can be further modified by
including projections of lagged variables.
Consider the following extension of  (\ref{proSEI}):
\begin{eqnarray}\label{proSE0}
\hskip-.6cm &X_t= \mu
+ \sum\limits_{s\le t} \zeta_s Q\Big(\alpha_{t-s}+   \sum\limits_{u=s+1}^{t-1}  \beta_{t-1-u, u-s} \left(\E_{[u,t-1]} X_{t-1} - \E_{[u+1,t-1]}\right) X_{t-1} \Big),
\end{eqnarray}
where  $\alpha_i, \beta_{i,j}, \, Q $ are the same as in (\ref{proSEI}) and the only  new feature is that $t$ is replaced
by $t-1$ in the inner sum on the r.h.s. of the equation. This fact allows to study {\it nonstationary } solutions
of (\ref{proSE0}) with a given projective initial condition  $X_t = X^0_t, t \le 0$ and the convergence
of $X_t$ to the equilibrium as $t \to \infty$; however, we shall
not pursue this topic in the present paper.
The following proposition is a simple extension
of Theorem \ref{thm1} and its proof is omitted.

\begin{proposition} \label{prop11}  Let $\alpha_i, \beta_{i,j}, Q$ satisfy the conditions of Theorem \ref{thm1}, including
(\ref{Qdom}) and (\ref{Kcond}).
Then there exists a unique solution $\{ X_t \} $ of (\ref{proSE0}), which is  written as
a projective moving average of (\ref{XMA})
with coefficients $g_{t-k,t}$ recursively defined as $g_{t-k,t} := Q(\alpha_k), k=0,1 $ and
\begin{eqnarray*}
g_{t-k,t}&:=&
Q\big(\alpha_k+ \sum_{i=0}^{k-2} \beta_{i,k-1-i} \zeta_{t-1-i} g_{t-1-i,t-1}  \big), \qquad k \ge 2.
\label{giter11}
\end{eqnarray*}
\end{proposition}

Finally, consider a projective equation (\ref{proSE}) with $\mu_t \equiv 0$
and kernels $Q_{s,t} $ $= $ $Q_{t-s} (x_{s+1,t-1}, \dots,$ $ x_{s+1,s})$ depending on $t-s$ real variables,
where $Q_0 = 1$ and
\begin{eqnarray} \label{Qj}
Q_j (x_1, \dots, x_j)&=&\frac{d(x_1)}{1} \cdot \frac{d(x_2) + 1}{2} \cdot \frac{d(x_3) + 2}{3}
\cdots
\frac{d(x_j) +j - 1}{j},
\end{eqnarray}
$j\ge 1$, where $d(x), x \in \R$ is a measurable function taking values in the interval $(-1/2, 1/2)$.
More explicitly,
\begin{eqnarray}\label{proseARFIMA}
\hskip-.8cm X_t&=&\sum_{j=0}^\infty   Q_j \big(\E_{[t-j+1,t-1]} X_{t-1}, \E_{[t-j+1,t-2]} X_{t-2}, \dots, \E_{[t -j+1,t-j]} X_{t-j} \big) \zeta_{t-j},
\end{eqnarray}
where $ \E_{[t -j+1,t-j]} X_{t-j} = \E X_t = 0$. Note that when $d(x) = d$ is constant,
$\{X_t \} $ (\ref{proseARFIMA})
is a stationary ARFIMA$(0,d,0)$ process. Time-varying fractionally integrated processes with
deterministic coefficients of the form
(\ref{Qj}) were studied in \cite{ref16a},  \cite{ref16b}.  We expect that  (\ref{proseARFIMA})
feature a ``random'' memory intensity depending on the values of the process. A rigorous study
of long memory properties of this model does not seem easy. On the other hand, solvability of
(\ref{proseARFIMA}) can be established similarly to  the previous cases (see below).

\begin{proposition} \label{prop1d}  Let $d(x)$ be a measurable function taking values in $(-1/2,1/2) $ and
such that $\sup_{x \in \R} d(x) \le \bar d$, where $\bar d \in (0,  1/2)$.
Then there exists a unique stationary solution $\{ X_t \} $ of (\ref{proseARFIMA}), which is  written as
a projective moving average of (\ref{XMA})
with coefficients $g_{s,t}$ recursively defined as $g_{t,t} := 1 $ and
\begin{eqnarray}
g_{s,t}&:=&
Q_{t-s}\big(\sum_{s< u \le t-1} \zeta_u g_{u, t-1}, \sum_{s< u \le t-2} \zeta_u g_{u, t-2}, \dots,  0\big), \qquad s<t,
\label{gARFIMA}
\end{eqnarray}
with $Q_{t-s} $ defined at (\ref{Qj}).
\end{proposition}

\begin{proof} Note that $\sup_{x_1, \dots, x_j \in \R} |Q_j(x_1,\dots, x_j)| \le \Gamma (\bar d +j)/\Gamma (\bar d) \Gamma(j) =:
\psi_j$ and $\sum_{j=0}^\infty \psi_j^2 < \infty$. Therefore
the $g_{s,t}$'s in (\ref{gARFIMA}) satisfy
$\sum_{s\le t} \E g^2_{s,t} < \infty $ for any $t \in \Z$. The rest of the proof is analogous as the case
of Theorem \ref{thm1}. \hfill $\Box$

\end{proof}

\section{Long memory}

In this section we study long memory properties (the decay of covariance and partial sums' limits) of
projective equations (\ref{proSEI})  and (\ref{proSEII}) in the case when the coefficients
$\alpha_j$'s decay slowly as $j^{d-1}, 0< d < 1/2.$

\begin{theorem} \label{LM1} Let $\{ X_t \} $ be the solution
of projective equation (\ref{proSEI}) satisfying the conditions of Theorem \ref{thm1} and $\mu = \E X_t = 0$.
Assume, in addition,
that $Q$ is a Lipschitz function, viz., there exists a constant $c_L >0$ such that
\begin{equation}\label{QLip}
|Q(x)- Q(y)| \ \ < \ c_L|x-y|, \qquad x, y \in \R
\end{equation}
and that there exist
$\kappa >0 $ and $0< d < 1/2$  such that
\begin{eqnarray}\label{mainterm}
b_j&:=&Q(\alpha_j) \ \sim \ \kappa \, j^{d-1}, \qquad j \to \infty
\end{eqnarray}
and
\begin{eqnarray}\label{barbeta1}
\bar \beta_j&:=&\max_{0\le i <j} |\beta_{i,j-i}| \ =\  o(b_j), \qquad j \to  \infty.
\end{eqnarray}
Then, as \ $t \to \infty $
\begin{eqnarray}\label{CovLM}
\E X_0 X_t&\sim&\sum_{k= 0}^\infty b_k b_{t+k}
 \ \sim \ \kappa^2_d  t^{2d-1},
\end{eqnarray}
where $\kappa^2_d :=  \kappa^2 B(d,1-d)$ and
$B(d,1-d)$ is beta function. Moreover, as $n \to \infty$
\begin{eqnarray}\label{FCLT}
n^{-1/2 -d}\sum_{t=1}^{[n\tau]}X_t&\longrightarrow_{D[0,1]}&c_{\kappa,d} B_H(\tau),
\end{eqnarray}
where $B_H$ is a fractional Brownian motion with parameter $H = d+ (1/2)$ and variance $\E B^2_H(t) = t^{2H}$
and $c^2_{\kappa,d} := \frac{\kappa^2 B(d,1-d)}{d(1+2d)}$.
\end{theorem}

\begin{proof}
Let us note that the statements (\ref{CovLM}) and (\ref{FCLT}) are well-known when
$\beta_{i,j} \equiv 0$, in which case $X_t$ coincides
with the linear process  $Y_t := \sum_{s\le t} b_{t-s} \zeta_s$. See, e.g.,
\cite{ref11}, Proposition 3.2.1 and Corollary 4.4.1.

The natural idea of the proof is to approximate $\{X_t\}$ by the linear process
$\{Y_t\}$. For $t \ge 0, k \ge 0$, denote
\begin{eqnarray*}
r^X_t&:=&\E X_0 X_t \ =  \
\sum_{s\le 0} \E [g_{s,0}\, g_{s,t}], \qquad r^Y_t \ := \  \E Y_0 Y_t \ =  \ \sum_{s \le 0} b_{-s} b_{t-s}, \\
\phi_{t-k,t}&:=&g_{t-k,t} - b_k \ = \  Q\Big(\alpha_k+ \sum_{i=0}^{k-1} \beta_{i,k-i} \zeta_{t-i} g_{t-i,t}  \Big)- Q(\alpha_k).
\end{eqnarray*}
Then
\begin{eqnarray*}
r^X_t- r^Y_t&=&
\sum_{s\le 0} \E [(b_{-s} +  \phi_{s,0}) (b_{t-s} + \phi_{s,t})  - b_{-s} b_{t-s}]\\
&=&\sum_{s\le 0} b_{-s} \E [\phi_{s,t} ]  + \sum_{s\le 0} b_{t-s} \E [\phi_{s,0} ] + \sum_{s\le 0} \E [\phi_{s,0} \, \phi_{s,t} ]
\ =: \ \sum_{i=1}^3 \rho_{i,t}.
\end{eqnarray*}
Using (\ref{QLip}) we obtain
\begin{eqnarray*}
|\E \phi_{t-k,t}|^2\ \le \ \E \phi^2_{t-k,t}
&\le&c^2_L  \E \Big(\sum_{i=0}^{k-1} \beta_{i,k-i} \zeta_{t-i} g_{t-i,t}  \Big)^2 \\
&=&c^2_L \Big(\sum_{i=0}^{k-1} \beta^2_{i,k-i} \E g^2_{t-i,t}  \Big) \\
&\le&\bar \beta^2_k c^2_L \Big(\sum_{i=0}^{\infty} \E g^2_{t-i,t}  \Big) \\
&\le&\bar \beta^2_k c^2_L K_{Q}.
\end{eqnarray*}
This and condition (\ref{barbeta1}) imply that
\begin{eqnarray*}
|\E \phi_{t-k,t}|  + \E^{1/2} \phi^2_{t-k,t}&\le&\delta_k  k^{d-1}, \qquad \forall \ t, k \ge 0
\end{eqnarray*}
where $\delta_k \to 0 \, (k \to \infty)$.
Therefore
for any $t \ge 1$
\begin{eqnarray*}
|\rho_{1,t}|&\le&C\sum_{k=1}^\infty k^{d-1} (t+ k)^{d-1} \delta_{t+k} \ \le \ C \delta'_t t^{2d-1}, \\
|\rho_{2,t}|&\le&C\sum_{k=1}^\infty k^{d-1} \delta_k (t+ k)^{d-1} \ \le \ C  \delta'_t t^{2d-1}, \\
|\rho_{3,t}|&\le&\sum_{s\le t} \E^{1/2} [\phi^2_{s,0}]  \, \E^{1/2}[ \phi^2_{s,t} ] \ \le \
C\sum_{k=1}^\infty k^{d-1} (t+ k)^{d-1}  \delta_k \delta_{t+k}\ \le \ C \delta'_t t^{2d-1},
 \end{eqnarray*}
where $\delta'_k \to 0 \ (k \to \infty)$. This proves (\ref{CovLM}).

To show (\ref{FCLT}), consider $Z_t := X_t - Y_t = \sum_{u\le t} \phi_{u,t} \zeta_u, t \in \Z. $
By stationarity of $\{Z_t \}$, for any $s\le t$ we have
$
\cov (Z_t, Z_s)=
\sum_{u\le 0}\E [\phi_{u,0} \, \phi_{u,t-s} ] \le \sum_{u\le 0}\E^{1/2} [\phi^2_{u,0}] \, \E^{1/2} [\phi^2_{u,t-s} ]$
$= o((t-s)^{2d-1}),$
see above, and therefore $ \E \big(\sum_{t=1}^n Z_t\big)^2 = o(n^{2d+1})$, implying
$$
n^{-d -(1/2)} \sum_{t=1}^{[n\tau]} X_t\ = \ n^{-d -(1/2)} \sum_{t=1}^{[n\tau]} Y_t + o_p(1).
$$
Therefore partial sums of $\{X_t\}$ and $\{Y_t\}$ tend to the same limit $c_{\kappa,d} B_H(\tau)$, in the sense of weak convergence
of finite dimensional distributions. The tightness in $D[0,1]$ follows from (\ref{CovLM}) and
the Kolmogorov criterion.  Theorem \ref{LM1} is proved. \hfill $\Box$
\end{proof}

\smallskip

A similar but somewhat different approximation by a linear process applies in the case of  projective equations of
(\ref{proSEII}).
Let us  discuss a special case of $\beta_{i,j}$:
\begin{equation}\label{betaK}
\beta_{i,j} = 1, \qquad \text{for all} \   i=0,1, \dots,  \, j=1,2,\dots.
\end{equation}
Note that for such $\beta_{i,j}$,
$\sum_{s<u\le t}  \beta_{t-u,u-s} (\E_{[u,t]} - \E_{[u+1,t]}) X_t =
\E_{[s+1,t]} X_t, \, s < t$
and the corresponding projective equation (\ref{proSEII}) with $\mu = 0, \alpha_i = b_i$ coincides
with (\ref{proma1}). Recall that  for bounded $\beta_{i,j}$'s as in (\ref{betaK}), condition
(\ref{QdomII}) on $Q$ together with $\sum_{i=0}^\infty \alpha^2_i < \infty $ guarantee
the existence
of the stationary solution $\{X_t\} $ (see Remark \ref{remMod}).  We shall also need the following
additional condition:
\begin{equation} \label{Qextra}
\E \big(Q(\E_{[s,0]} X_0) -  Q(X_0)\big)^2 \ \to \ 0, \qquad \text{as} \quad s \to - \infty.
\end{equation}
Since $\E \big(\E_{[s,0]} X_0 - X_0\big)^2 \to 0,  \, s \to - \infty$, so (\ref{Qextra}) is satisfied
if $Q$ is Lipschitz, but otherwise conditions (\ref{Qextra}) and (\ref{QdomII}) allow $Q$ to be even
discontinuous.
Denote
$$
c^2_{Q,d} := \big(\E [Q (X_0)] \big)^2  B(d,1-d).
$$

\begin{theorem} \label{LM2} Let $\{ X_t \} $ be the solution
of projective equation (\ref{proSEII}) with $\mu = 0, \, \beta_{i,j}$ as in (\ref{betaK}),
$Q$ satisfying (\ref{QdomII})
and
\begin{equation}\label{alfaLM}
\alpha_k\ \sim\  \, k^{d-1}, \qquad k \to \infty, \qquad \exists \  \ 0< d < 1/2.
\end{equation}
In addition, let (\ref{Qextra}) hold.
Then
\begin{equation}\label{CovLMII}
\E X_0 X_t\ \sim\  c^2_{Q,d} t^{2d-1}, \qquad t \to \infty
\end{equation}
and
\begin{eqnarray}\label{FCLTII}
n^{-1/2 -d}\sum_{t=1}^{[n\tau]}X_t&\longrightarrow_{D[0,1]}&c'_{Q,d} B_H(\tau),  \qquad
c'_{Q,d} := c_{Q,d}/(d(1+2d)^{1/2}.
\end{eqnarray}
\end{theorem}

\begin{proof}
Similarly as in the proof of the previous theorem,
let
$Y_t := \sum_{s\le t} b_{t-s} \zeta_s$,  $b_k := \ \alpha_k \E [Q ( X_0)]$, \,  $r^X_t:=\E X_0 X_t, \,
r^Y_t:=\E Y_0 Y_t, \, t\ge 0$.  Relation (\ref{CovLMII}) follows from
\begin{equation}\label{rXY}
r^X_t - r^Y_t \ = \ o(t^{2d-1}), \qquad t \to \infty.
\end{equation}
We have $X_t = \sum_{s\le t} g_{s,t} \zeta_s, \, g_{s,t} = \alpha_{t-s} Q(\E_{[s+1,t]} X_t), \, \E X^2_t = \sum_{s\le t} \E g^2_{s,t} < \infty $ and
$\E [Q(\E_{[s+1,t]} X_t)^2] \le c_0^2 + c_1^2 \E (\E_{[s+1,t]} X_t)^2 \le c_0^2 + c_1^2 \E X^2_t < C.
$
Decompose $r^X_t = r^X_{1,t} + r^X_{2,t}$, where
\begin{eqnarray*}
r^X_{1,t}
&:=&\sum_{s\le 0} \alpha_s \alpha_{t+s} \E [Q(\E_{[s+1,0]} X_0)] \,\E[ Q(\E_{[1,t]} X_t)  ], \qquad
r^X_{2,t} \ := \ \sum_{s\le 0} \alpha_s \alpha_{t+s} \gamma_{s,t},
\end{eqnarray*}
and where
\begin{eqnarray*}
|\gamma_{s,t}|&:=&\big|\E \big[Q(\E_{[s+1,0]} X_0) \big\{ Q(\E_{[s+1,t]} X_t) - Q(\E_{[1,t]} X_t)\big\}  \big]\big|
\ \le \ \tilde \gamma^{1/2}_{1,s} \tilde \gamma^{1/2}_{2,s,t},
\end{eqnarray*}
Here, $\tilde \gamma_{1,s} :=  \E[Q^2(\E_{[s+1,0]} X_0)] \le C, $ see above, while
\begin{eqnarray}
|\tilde \gamma_{2,s,t}|&:=&\E \big[\big(Q(\E_{[s+1,t]} X_t) - Q(\E_{[1,t]} X_t)\big)^2 \big]\nn \\
&=&\E \big[\big(Q(\E_{[s+1-t,0]} X_0) - Q(\E_{[1-t,0]} X_0)\big)^2 \big]\ \to \ 0,  \qquad t \to \infty
\label{Qgamma}
\end{eqnarray}
uniformly in $s \le 0$, according to (\ref{Qextra}). Hence and from
(\ref{alfaLM})
it follows that
\begin{equation}\label{rliek}
|r^X_{2,t}|
\ = \  o(t^{2d-1}), \qquad t \to \infty.
\end{equation}
Accordingly, it suffices to prove (\ref{rXY}) with $r^X_t$ replaced by $r^X_{1,t}$. We have
\begin{eqnarray*}
r^X_{1,t}
&=&r^Y_t + \sum_{s\le 0} \alpha_s \alpha_{t+s}
\phi_{1, s,t}
+ \sum_{s\le 0} \alpha_s \alpha_{t+s}
\phi_{2, s,t} + \sum_{s\le 0} \alpha_s \alpha_{t+s}
\phi_{3, s,t},
\end{eqnarray*}
where the ``remainders'' $\phi_{1, s,t} := \E [Q(X_0)] \big\{\E [Q(\E_{[s+1,0]} X_0)] - \E [Q(X_0)]\big\}$, $\
\phi_{2, s,t} :=$ $\E [Q(X_0)]$
$\times  \big\{\E [Q(\E_{[1-t,0]} X_0)] - \E [Q(X_0)]\big\} $ and  $ \phi_{3, s,t} := \big(\E [Q(\E_{[s+1,0]} X_0)] - \E [Q(X_0)]\big)\big(\E [Q(\E_{[1-t,0]} X_0)] - \E [Q(X_0)]\big) $
can be estimated similarly to (\ref{Qgamma}), leading to the asymptotics of (\ref{rliek}) for
each of the three sums in the above decomposition of $r^X_{1,t}$. This proves (\ref{CovLMII}).

\smallskip

Let us prove (\ref{FCLTII}). Consider the convergence of one-dimensional distributions for $\tau = 1$, viz.,
\begin{equation}\label{clt2}
n^{-d-1/2} S^X_n \ \to \ {\cal N}(0, \sigma^2), \qquad \sigma = c'_{Q,d}
\end{equation}
where $S^X_n := \sum_{t=1}^n X_t$. Then (\ref{clt2}) follows from
\begin{equation}\label{clt2Z}
\E (S^X_n - S^Y_n)^2 \ =  \  o(n^{1+ 2d}),
\end{equation}
where $S^Y_n := \sum_{t=1}^n Y_t$ and $Y_t$ is as above.  We have
\begin{eqnarray}\label{SXY}
 \E (S^X_n - S^Y_n)^2&=&\E \Big( \sum_{s\le n} \zeta_s \sum_{t=1\vee s}^n \alpha_{t-s} \tilde  Q_{s,t}  \Big)^2 \nn \\
 &= &
\sum_{s \le n} \sum_{t_1, t_2=1\vee s}^n \alpha_{t_1-s} \alpha_{t_2-s} \E [ \tilde  Q_{s,t_1} \tilde Q_{s,t_2}],
\end{eqnarray}
where $\tilde Q_{s,t} := Q\big(\E_{[s+1,t]} X_t \big)-
\E [Q (X_0)]$. Let us prove that uniformly in $s \le t_1 $
\begin{eqnarray} \label{Q?}
\E [\tilde  Q_{s,t_1} \tilde  Q_{s,t_2}]&=&o(1), \qquad \text{as} \qquad t_2-t_1 \to \infty.
\end{eqnarray}
We have for $s \le t_1 \le t_2$ that
\begin{eqnarray*}
\E [\tilde  Q_{s,t_1} \tilde  Q_{s,t_2}]&=&\E \Big[\tilde  Q_{s,t_1} \big\{Q\big(\E_{[s+1,t_2]} X_{t_2} \big)-
\E [Q (X_0)]\big\} \Big] \\
&=&\E [\tilde  Q_{s,t_1}] \big\{ \E \big[Q\big(\E_{[t_1+1,t_2]} X_{t_2} \big)\big] -
\E [Q (X_0)] \big\}\\
&+&\E \Big[\tilde  Q_{s,t_1} \big\{ Q\big(\E_{[s+1,t_2]} X_{t_2} \big)- Q\big(\E_{[t_1+1,t_2]} X_{t_2} \big)\big\} \Big] \
=: \ \psi'_{s,t_1,t_2} + \psi''_{s,t_1,t_2},
\end{eqnarray*}
where we used the fact that $\tilde  Q_{s,t_1}$ and $Q\big(\E_{[t_1+1,t_2]} X_{t_2} \big)$ are independent. Here,
thanks to  (\ref{Qextra}), we see that
$|\psi'_{s,t_1,t_2}| \le \E^{1/2} [\tilde  Q^2_{s,t_1}] \,\E^{1/2} \big[\big\{Q\big(\E_{[t_1-t_2+1,0]} X_{0}\big) - Q(X_0)\big\}^2\big]
\le C \E^{1/2} \big[\big\{Q\big(\E_{[t_1-t_2+1,0]} X_{0}\big) - Q(X_0)\big\}^2\big]$ $ \to 0$ uniformly in $s \le t_1 \le t_2$ as
$t_2 - t_1 \to \infty $.  The same is true for $|\psi''_{s,t_1,t_2}|$ since
it is completely analogous to (\ref{Qgamma}). This proves (\ref{Q?}). Next,
with (\ref{SXY}) in mind, split $\E (S^X_n - S^Y_n)^2 =: T_n =  T_{1,n} + T_{2,n} $, where
$$
T_{1,n} := \sum_{s \le n} \sum_{t_1, t_2=1\vee s}^n   \1(|t_1-t_2| >K) \dots,
\qquad  T_{2,n} :=
\sum_{s \le n} \sum_{t_1, t_2=1\vee s}^n \1 (|t_1-t_2|\le K) \dots,
$$
where $K$ is a large number. By (\ref{Q?}), for any $\epsilon>0 $ we can find $K>0$ such that
$\sup_{s \le t_1 < t_2: t_2 -t_1 > K} $ $|\E [\tilde  Q_{s,t_1} \tilde  Q_{s,t_2}]| < \epsilon $ and therefore
$$
|T_{1,n}|\ <\ \epsilon \sum_{s \le n} \sum_{t_1, t_2=1\vee s}^n
|\alpha_{t_1-s} \alpha_{t_2-s}|\ \le \ C\epsilon \sum_{t_1, t_2=1}^n |\bar r_{t_1-t_2}| \
\le \ C \epsilon n^{1+2d}
$$
holds for all $n>1$ large enough, where $\bar r_t := \sum_{i=0}^\infty |\alpha_i \alpha_{t+i}| = O(t^{2d-1})$ in view
of (\ref{alfaLM}).
On the other hand,
$|T_{2,n}| \le C K n = o(n^{1+2d})$ for any $ K < \infty $ fixed. Then (\ref{clt2Z}) follows, implying
the finite-dimensional convergence in (\ref{FCLTII}). The tightness in (\ref{FCLTII}) follows from (\ref{CovLMII}) and
the Kolmogorov
criterion, similarly as in the proof of Theorem \ref{LM1}. Theorem \ref{LM2} is proved. \hfill $\Box$
\end{proof}

\begin{remark}\label{Wu0}
{\rm Shao and Wu \cite{ref17} discussed partial sums limits of fractionally integrated nonlinear processes
$Y_t = (1- L)^{-d} u_t, \, t \in \Z,  $ where $L X_t = X_{t-1}$ is the backward shift, $(1- L)^d = \sum_{j=0}^\infty \psi_j (d) L^j, \,
d \in (-1,1)$ is the fractional differentiation operator,
and
$\{u_t \} $ is a causal Bernoulli shift:
\begin{equation}\label{ut}
u_t = F(\dots, \zeta_{t-1}, \zeta_t), \qquad t \in \Z
\end{equation}
in i.i.d. r.v.'s $\{\zeta_t, t \in \Z \} $. The weak dependence condition
on $\{u_t\} $  (\ref{ut}), analogous to (\ref{DedMer}) and
guaranteeing the weak convergence of normalized partial sums of $\{Y_t \}$ towards a fractional Brownian motion,
is written in terms of projections $P_0 u_t = (\E_{[0,t]} - \E_{[1,t]})u_t $:
\begin{eqnarray}\label{Wu}
\Omega(q) \ := \  \sum_{t=1}^{\infty} \|P_0 u_t\|_q \ < \ \infty,
\end{eqnarray}
where $\| \xi \|_q := \E^{1/q} |\xi|^q $ and $q= 2$ for $0< d < 1/2$;
see (\cite{ref17}, Thm.~2.1), also \cite{ref21}, \cite{ref20}. The above mentioned papers verify
(\ref{Wu}) for several classes of  Bernoulli shifts.  It is of interest
to verify (\ref{Wu}) for projective moving averages.
For $X_t $ of (\ref{proma0}) and
$0< d < 1/2$,
$u_t := (1-L)^d X_t
= \sum_{s \le t} \zeta_s G_{s,t}$ is a well-defined projective
moving average with coefficients
$$
G_{s,t} := \sum_{s\le   v \le t} \psi_{t-v}(d) g_{s, v}, \qquad s \le t,
$$
see Proposition \ref{propfilter}.
For concreteness, let $g_{s,t} = \psi_{t-s}(-d) Q(\E_{[s+1,t]} X_t) $ as in Theorem \ref{LM2} with
$\alpha_j = \psi_j(-d)$. We have
 $\Omega(2) = \sum_{t=1}^{\infty} \|G_{0,t}\|_2 $, where
\begin{eqnarray}
\hskip-1.5cm &&\|G_{0,t}\|^2_2
= \E\Big[ \sum_{v=0}^t
\psi_{t-v}(d) \psi_{v}(-d) Q(\E_{[1,v]} X_v) \Big]^2\
 =  \E \Big[ \sum_{v=0}^{t-1}
\psi_{t-v}(d) \psi_{v}(-d)  
Q_{v,t}\Big]^2, 
\label{Wu1}
\end{eqnarray}
where $Q_{v,t} := Q(\E_{[1,v]} X_v) - Q(\E_{[1,t]} X_t)$
and we used $ \sum_{v=0}^t
\psi_{t-v}(d) \psi_{v}(-d) = 0, \, t \ge 1$ in the last equality.
Note that  $ \psi_{t-v}(d) \psi_{v}(-d)  < 0$ have the same sign and $Q_{v,t} \approx Q(X_{v}) - Q(X_t)$ are not negligible in
(\ref{Wu1}). Therefore
we conjecture that  $\|G_{0,t}\|^2_2 = O\big(\sum_{v=0}^{t-1}
|\psi_{t-v}(d) \psi_{v}(-d)|\big)^2$  $   =  O(t^{-2(1-d)})$ and hence
$\Omega(2) = \infty $ for $0< d < 1/2$. The above argument suggests that projective
moving averages posses a different "memory mechanism" from
fractionally integrated processes in \cite{ref17}.

}
\end{remark}

\section{Acknowledgements}

This work was supported by a grant (No.\ MIP-063/2013) from the Research Council of Lithuania.
The authors also thank an anonymous referee for useful  remarks.


\end{document}